\documentclass{amsart}
\usepackage{amsmath}
\usepackage{paralist}
\usepackage{url}
\usepackage{hyperref}
\usepackage{color}
\usepackage{framed}
\usepackage{mathrsfs}
\usepackage{graphicx}
\usepackage{bbm}
\newcommand{\R}[1]{{\mathbb{R}^{#1}}}
\newcommand{\sm}[1]{{\mathcal{S}^{#1}}}
\newcommand{\cp}{{C^o}}
\newcommand{\RR}{\mathbb{R}}
\newcommand{\rccp}{{0^+C^o}}
\newcommand{\CC}{{0^+C^o}}
\newcommand{\pointed}{C_0}
\newcommand{\rc}{{0^+\pointed}}

\newcommand{\ks}{{K^*}}

\newcommand{\svpc}{{S_{C}}}
\newcommand{\svpcone}{{S_{C}^1}}
\newcommand{\svpctwo}{{S_{C}^2}}
\newcommand{\svpcthree}{{S_{C}^3}}

\newcommand{\svpci}{{S_{c}^i}}

\newcommand{\svprci}{{S_{0^+C}^i}}
\newcommand{\rpc}{{0^+C}}

\newcommand{\lin}{L_0}

\newcommand{\svcone}{{S_{C_0}^1}}
\newcommand{\svctwo}{{S_{C_0}^2}}
\newcommand{\svcthree}{{S_{C_0}^3}}

\newcommand{\svrci}{{S_{0^+\pointed}^i}}
\newcommand{\svci}{{S_{C_0}^i}}
\newcommand{\svc}{{S_{C_0}}}

\newcommand{\svl}{{S_{\linc}}}

\newcommand{\linc}{{L_1}}
\newcommand{\linco}{{L_1^\bot}}
\newcommand{\lincc}{{L_2}}
\newcommand{\lincco}{{L_2^\bot}}

\newcommand{\cl}[1]{\mathrm{cl}\left(#1\right)}
\newcommand{\ri}[1]{\mathrm{ri}\left(#1\right)}
\newcommand{\inte}[1]{\mathrm{int}\left(#1\right)}
\newcommand{\co}[1]{{\mathrm{co}}\left(#1\right)}
\newcommand{\cone}[1]{{\mathrm{cone}}\left(#1\right)}
\newcommand{\support}[2]{\delta^*\left(#1,#2\right)}
\newcommand{\indicator}[2]{\delta \left(#1,#2\right)}
\newcommand{\domsupport}[2]{\mathrm{dom}\support{#1}{#2}}
\newcommand{\extc}[1]{\mathrm{ext}_1(#1)}
\newcommand{\extrc}[1]{\mathrm{ext}_2(#1)}
\newcommand{\aff}[1]{\mathrm{aff}(#1)}

\newcommand{\pc}{C^{\text{\upshape o}}}
\newcommand{\dpc}{C^{\text{\upshape oo}}}

\newcommand{\rec}{0^+C}
\newcommand{\pk}{K^{\text{\upshape o}}}

\newcommand{\CA}{C^{\text{\upshape o}}}

\newtheorem{theorem}{Theorem}[section]

\newtheorem{lemma}[theorem]{Lemma}

\newtheorem{corollary}[theorem]{Corollary}

\theoremstyle{definition}
\newtheorem{example}[theorem]{Example}

\newtheorem{remark}[theorem]{Remark}
\newtheorem{define}[theorem]{Definition}

\numberwithin{equation}{section}

\begin{document}

\title[Lifts of Non-compact Convex Sets and Cone Factorizations]
{Lifts of Non-compact Convex Sets and Cone Factorizations }

\author{Chu Wang}
\address{Key Laboratory of Mathematics Mechanization, Academy of Mathematics and Systems Science, CAS,
Beijing, 100190, China}
\email{cwang@mmrc.iss.ac.cn}
\author{Lihong Zhi}
\address{Key Laboratory of Mathematics Mechanization, Academy of Mathematics and Systems Science, CAS,
Beijing, 100190, China}
\email{lzhi@mmrc.iss.ac.cn}

\begin{abstract}

In this paper we   generalize the factorization theorem of Gouveia, Parrilo and Thomas 
 to a broader class of convex sets.  Given a general convex set, we define  a slack operator associated to the set and its polar according to whether  the convex set is full dimensional, whether it is a translated cone and whether it contains lines. We strengthen the condition of a cone lift by requiring not only the convex set is the image of an affine slice of a given closed convex cone, but also its recession cone is the image of the linear slice of the closed convex cone. We show that the generalized lift of a convex set  can also be characterized by the cone factorization of a properly defined slack operator.
\end{abstract}
\date{\today}
\maketitle


\section{Introduction}
 Given a linear programming problem, how to reformulate it to a  standard form with fewer constraints is an important problem. In \cite{Yannakakis1991},  Yannakakis proved that the nonnegative rank of a slack matrix of a polytope $P$ is the minimum $k$ such that $P$ is the linear image of an affine slice of the nonnegative quadrant.  In \cite{Fiorini2012, Gouveia2013}, Yannakakis's result was generalized to decide
 whether  a convex body $C$ (a compact convex set containing the origin in its interior) is the linear image of an affine slice of a given convex cone $K$ ($K$-lift) via cone factorizations of slack operators. Although it was  claimed that  results in   \cite{Gouveia2013} hold for all convex sets,  we notice that it is more complicated to identify whether a non-compact convex set   $C$ containing no lines has a $K$-lift since $C$ could be generated by not only  extreme points but also extreme directions. Moreover, if a convex set contains lines, then it has no extreme points or extreme directions. Furthermore, if the convex set  does not  contain the origin  in its interior,  linear functions corresponding to its polar can not characterize the convex set completely (see Example \ref{nonequal}).  These facts  motivate us to study how to extend the definitions of $K$-lift and slack operator to a general convex set and show the relationship between  lifts of convex sets and cone factorizations of  slack operators when the convex set is not a convex body.

 {\bf Our contribution:} Let $C$ be a closed convex set in $\R{n}$ and $K$ a closed convex cone in $\R{m}$. We consider how to  generalize the factorization theorem in   \cite{Gouveia2013}  to a broader class of convex sets. Our main results are as follows.
\begin{itemize}
\item When $C$  contains the origin, since $C=C^{oo}$, $C$ can be described by all vectors in $\cp$. When $C$ does not contain the origin, we show that a convex set $C$ can be  characterized completely by linear functions defined by  $\cp=\{l\mid \langle l,x\rangle\leq 1,\ \forall x\in  C\}$, $\rccp=\{l\mid \langle l,x\rangle\leq 0,\ \forall x\in C \}$ and  $C_3=\{l\mid \langle l,x\rangle\leq -1,\ \forall x\in C \}$.
\item We extend the question of when a given convex body $C$ is the linear image of  an affine slice of a  convex cone   to the case where $C$ is not compact and  may  not contain the origin in its interior and may contain lines. We introduce two ways to characterize the existence of a $K$-lift of $C$. The first one is based on all  points in $C$ and the second one  uses only  extreme points, extreme directions and  an orthogonal basis of the linearlity space of  $C$ if it contains  lines.
    Although the first method can be used to check the existence of a $K$-lift of any convex set, it is difficult to use, see Remark \ref{introduction}. Therefore, in the paper,  we focus on the second method. We extend Definition 1,2 and Theorem 1 in \cite{Gouveia2013} to a broader class of convex set and show that the generalized lift of a convex set can also be characterized by the cone factorization of properly defined slack operator according to  whether the convex set  is full dimensional, whether it is a translated cone and whether it contains a line.
\item We specialize the results of the cone lift of general convex sets to  polyhedra   and  show that the conclusion can be strengthened when  $C$ and  $K$ are both polyhedra.  When $K$ is a semidefinite convex cone, we   give a lower bound on the semidefinte rank of a polyhedron, which generalizes the result in \cite{GRT2013}.  We also extend  results in \cite{ Gillis2012,Gouveia20132921, GRT2013} to  identify  whether a given nonnegative matrix is a slack matrix of a polyhedron and characterize the rank of a  slack matrix in terms of the dimension of  a polyhedron.
\end{itemize}

 The paper is organized as follows. In Section \ref{premli}, we provide some preliminaries about convex sets and cones.  Some well-known results in convex analysis are recalled. In Section \ref{noncompactlift}, we generalize the factorization theorem in \cite[Theorem 1]{Gouveia2013} to  convex sets which are not  convex bodies. In Section \ref{polyhedralift}, we specialize  results established in Section \ref{noncompactlift}  to the case where the convex set is a polyhedron. Some results in \cite{Gillis2012,Gouveia20132921,GRT2013} on the semidefinite rank of a slack matrix are extended to the case that the convex set is a polyhedron.

\section{Preliminary}\label{premli}

 Let $\R{n}$ be a n-dimensional linear space,  $\sm{m}$ the space of real symmetric $m\times m$ matrices.  A non-empty subset $C\subset\R{n}$ is said to be {\itshape convex} if
$(1-\lambda)x+\lambda y\in C$ whenever $x\in C$, $y\in C$ and $0<\lambda<1$.  We
denote $\cl{C}$ and $\inte{C}$ as the {\itshape closure} and {\itshape interior} of $C$ respectively. The
{\itshape affine
hull} of a convex set $C$, denoted by  $\aff{C}$, is the unique smallest affine set containing $C$. If a closed convex set $C$ is compact and contains the origin in its interior, it is called a {\itshape convex body}.

A subset $K$ of $\R{n}$ is called a {\it cone} if it is closed under nonnegative scalar multiplication, i.e. $\lambda x\in K$ when $x\in K$ and $\lambda\geq 0$.  We denote the $m$-dimensional nonnegative quadrant by $\RR_{+}^{m}$ and the cone of $m \times m$ real symmetric positive semidefinite (psd) matrices by $\mathcal{S}_{+}^m$.
A convex cone $K$ is {\itshape pointed} if it is closed and
$K\cap-K=\{0\}$.
The {\itshape polar}  
 of  a non-empty convex cone $K$ is defined as
\begin{equation*}
\pk=\{x\in\R{n}\mid \forall y\in K,\langle x,y\rangle\leq 0 \}.
\end{equation*}
Given a  set $C$, if there exists a cone $C_0$ and a vector $x\in \R{n}$ such that $C=x+C_0$, then $C$ is said to be a {\it translated cone}.

The
{\itshape recession cone} $\rec$ of a non-empty  convex set $C$ is the set including  all
vectors  $y$ satisfying  $x+\lambda y\in C$ for every $\lambda>0$ and
$x\in C$.  The set $\rec \cap (-\rec)$ is called the {\it lineality space} of $C$.

 Let $S_0$ be a set of points in $\R{n}$ and $S_1$ a set of directions in $\R{n}$. We define the {\it convex hull} $\co{S}$ of $S=S_0\cup S_1$ to be the smallest convex set $C$ in $\R{n}$ such that $C\supseteq S_0$ and $\rpc\supseteq S_1$. Algebraically, a vector $x$ belongs to $\co{S}$ if and only if it can be expressed in the form
 \[x=\lambda_1x_1+\cdots+\lambda_kx_k+\lambda_{k+1}x_{k+1}+\cdots+\lambda_{m}x_{m},\  \sum\limits_{i=1}^{k}\lambda_i=1,\]
 where $x_1,\ldots,x_k$ are vectors in $S_0$ and $x_{k+1},\ldots,x_{m}$ are vectors whose directions are in $S_1$ and $\lambda_i\geq 0$ for $1\leq i\leq m$. If $S_0=\{0\}$ and $S_1$ is not empty, then $C=\co{S_0\cup S_1}$ is a cone which is also denoted as $\cone{S_1}$.

  The {\it relative interior} $\ri{C}$ of a convex set $C$ in $\R{n}$ is defined as the interior  when $C$ is regarded as a subset of its affine hull $\aff{C}$. A {\it face} of a convex set $C$ is a convex subset $C'$ of $C$ such that every closed line segment in $C$ with a relative interior 
in $C'$ has both endpoints in $C'$. The zero-dimensional faces of $C$ are called the {\it extreme points} of $C$. If $C'$ is a half-line face of a convex set $C$, we shall call the direction of $C'$ an {\it extreme direction} of $C$. If $C$ is a convex cone, an {\it extreme ray} is a face  which is a half-line emanating from the origin. Note that every extreme direction of $C$ can also be regarded as an extreme ray  of $\rpc$. Let $x=(x_1,\ldots,x_n)$ and $y=(y_1,\ldots,y_n)$ be two vectors in $\R{n}$, the  {\it inner product} of   $x, y$ in $\R{n}$ is expressed by   $\langle x,y\rangle=\sum\limits_{i=1}^n x_i y_i$.

The {\itshape polar} of a non-empty convex set $C\subset\R{n}$ is a closed
convex set defined as
\begin{equation*}
\pc=\{x\in\R{n}\mid \forall y\in C,\langle x,y\rangle\leq 1 \}.
\end{equation*}
We have $\dpc=\cl{\co{C\cup\{0\}}}$.

The {\it indicator function} $\indicator{\cdot}{C}$ is defined by
\begin{equation*}
\indicator{x}{C}= \left\{ \begin{array}{lll} 0 &~{\rm if}~ &x \in C,\\
+\infty & ~{\rm if}~ &x \notin C.
\end{array} \right.
\end{equation*}
The {\it support function}
 $\support{x}{C}$ of a convex set $C \in \R{n}$ is defined by
 \[\support{x}{C}=\sup \{\langle x, y\rangle\mid y\in C\}.\]
$\domsupport{x}{C}=\{x\mid \support{x}{C}<+\infty\}$ is called the {\it  barrier cone} of $C$.

\begin{theorem}\cite[Theorem 8.3]{convexanalysis}\label{closedone}
Let $C$ be a non-empty closed convex set, and let $y\neq0$. If there exists even one $x$ such that the half-line $\{x+\lambda y\mid \lambda \geq 0\}$ is contained in $C$, then the same thing is true for every $x\in C$, i.e. one has $y\in \rpc$.
\end{theorem}

\begin{theorem}\cite[Theorem 18.5]{convexanalysis}\label{noncomexp}
Let $C$ be a closed convex set containing no lines, and let $S$ be the set of all extreme points and extreme directions of $C$. Then $C= \co{S}$.
\end{theorem}

\begin{theorem}\cite[Theorem 8.7]{convexanalysis}\label{structure}
Let $f$ be a closed proper convex function. Then all the non-empty level sets of the form $\{x \mid f(x)\leq \alpha\}$, $\alpha \in \mathbb{R}$,  have the same recession cone and the same lineality space.
\end{theorem}

\begin{corollary}\cite[Corollary 14.2.1]{convexanalysis}\label{dualpolar}
The polar of the barrier cone of a non-empty closed convex set $C$ is the recession cone of $C$.
\end{corollary}
\begin{theorem}\cite[Theorem 13.1]{convexanalysis}\label{charc}
Let $C$ be a convex set. Then $x\in\cl{C}$ if and only if $\langle x,x^*\rangle\leq \support{x}{C}$ for every vector $x^*$.
\end{theorem}

\section{Cone lifts of non-compact convex sets}\label{noncompactlift}
Let $\extc{C}$ denote the set of  extreme points of a closed convex set  $C$ and  $\extrc{C}$  the set of  extreme rays of a closed convex cone $C$.
An extreme ray is also the common direction of   vectors in this ray. In the following part of our paper, we  represent each extreme ray by one vector and denote $\extrc{C}$ as the collection of such vectors.

If $C$ is a compact convex set containing the origin in its interior,    according to
 \cite[Definition 1]{Gouveia2013},  a {\itshape $K$-lift} of $C\subset\R{n}$ is a set $Q=K\cap L$ where $L\subset\R{m}$  is an affine subspace and $\pi:\R{m}\rightarrow\R{n}$ is a linear map  such that
\begin{equation}\label{CKlift}
C=\pi(K\cap L).
\end{equation}
 If $L$ intersects the interior of $K$, we say that $Q$ is a {\itshape proper $K$-lift} of $C$. The slack operator $\svpc$ is {\it$K$-factorizable} if there exists maps
\begin{equation*}
A:\extc{C}\rightarrow K, \ B: \extc{\cp}\rightarrow \ks.
\end{equation*}
such that $\svpc(x,y)=\langle A(x),B(y)\rangle$ for all $(x,y)\in \extc{C}\times\extc{\cp}$, see \cite[Definition 2]{Gouveia2013}.

In this section, we explain how to generalize the  argument in \cite{Gouveia2013} to more general convex sets and show the relationship between  cone lifts of convex sets and  cone factorizations of  slack operators when the closed convex set is not a convex body.

\subsection{ $C$ is  full dimensional}
Assume  that $C$ is a full dimensional closed  convex set in $\R{n}$, we define
\begin{equation*}
\cp=\{x\mid \support{x}{C} \leq 1\}, \, \rccp=\{x\mid \support{x}{C} \leq 0\}, \, C_3=\{x\mid \support{x}{C} \leq -1\}.
\end{equation*}
   It is clear  that $\CA$, $C_3$ are closed convex sets containing no lines and  $\rccp$ is a closed pointed cone that contains $C_3$. Let
\begin{equation*}
D_1=\extc{\cp} \backslash 0,\,
  D_2=\extrc{\rccp} \cap \{x\mid\support{x}{C}=0\},\,
    D_3=\extc{C_3}.
\end{equation*}
By Theorem \ref{structure}, we have $\CC=0^+C_3$.
Let
\begin{equation*}
D_{32}=\extrc{\rccp} \cap \{x\mid\support{x}{C}=-1\}.
\end{equation*}
It is clear that $D_{32}\subseteq D_3$ but $D_{32}$ is not always equal to $D_3$.

\begin{example}\label{nonequal}
 We consider a compact  convex set
 \[C=\{(x,y)\mid x+y\geq 1, x+y\leq 3, \ y-x\geq -1,\ y-x\leq 1.\}\]
  Then, we have
  \begin{eqnarray*}
  \cp &=&\{(x,y)\mid 2x+y\leq 1,\ x+2y\leq 1,\ x\leq 1,\ y\leq 1\},\\
   \CC &=& \{(x,y)\mid x\leq 0,\ y\leq 0\},\\
   C_3 &=& \{(x,y)\mid x\leq -1,\ y\leq -1\},
   \end{eqnarray*}
    see  Figure \ref{fig1}.  Furthermore, we have $D_{32}=\emptyset$ and
    \[
    D_1 =\{(-1,1),(\frac{1}{3}, \frac{1}{3}), 
    (1,-1)\},\,
     D_2=\{(-1,0),(0,-1)\}, \, D_3=\{(-1,-1)\}.\]
 \begin{figure}[!htb]
\includegraphics[width=1\textwidth]{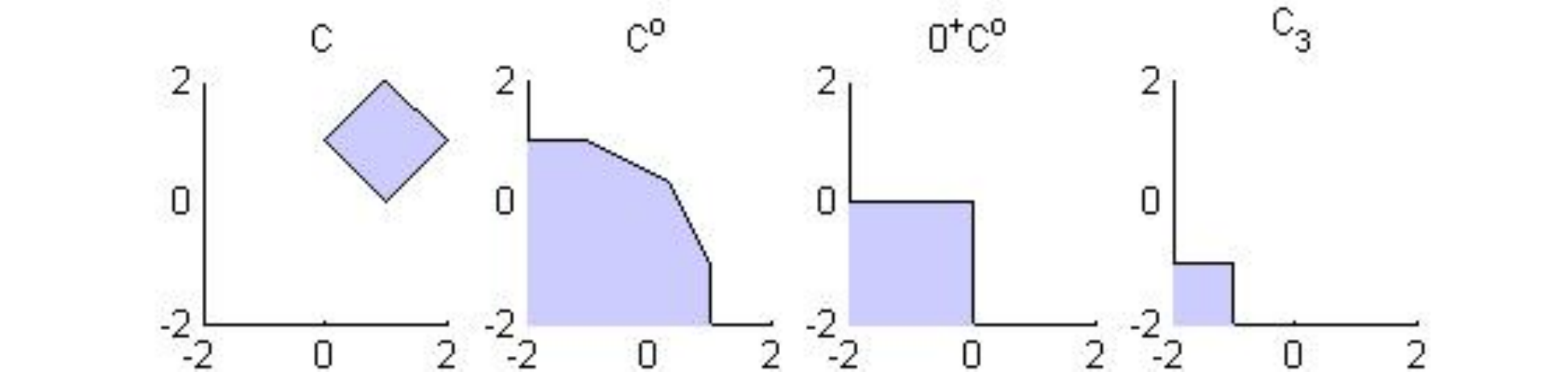}
\caption{Example \ref{nonequal} }\label{fig1}
\end{figure}
\end{example}

\begin{remark}\label{D123}
According to Theorem \ref{noncomexp} and \ref{structure}, we can show:
 \begin{enumerate}
 \item The convex set $\cp$ can be expressed
as convex combinations of  all points in $D_1$ and all directions  of the vectors in $D_2$ and $D_{32}$.
\item The convex cone $\rccp$ can be expressed  as convex combinations of all directions of the vectors in $D_2$ and $D_{32}$.
\item The convex set $C_3$ can be  expressed
as convex combinations of  all points in $D_3$ and all directions of  the vectors  in $D_2$ and $D_{32}$.
\end{enumerate}
\end{remark}

\begin{theorem}\label{small lemma}
Given a full dimensional closed convex set $C\subset\R{n}$,  the following statements are true:
\begin{enumerate}
\item The  set $D_1$ is empty if and only if $C^o$ is a closed cone. If $D_1$ is not empty, for every vector $x$ in $D_1$, we have $\support{x}{C}=1$.
\item The set  $D_2$  is empty if  $C$ contains the origin in its interior.  When $C$ is not compact and contains the origin on its boundary, 
    $D_2$ is not empty and each extreme ray of $\rccp$ is the direction of a vector in $D_2$.
\item  The  set  $D_3$ is empty if and only if $C$ contains the origin. If $D_3$ is not empty, for every vector $x$ in $D_3$, $\support{x}{C}=-1$.
\end{enumerate}
Moreover, the convex cone generated by $\cp$ is $\domsupport{x}{C}$.
\end{theorem}

\begin{proof}

 Since $\cp$  contains no lines, $D_1$ is empty if and only if the origin is the only extreme point of $\cp$, i.e.  $\cp$ is a closed cone.
If there exists an extreme  point $x\in D_1$ such that $\support{x}{C}<1$,  then there exists $\lambda>0$ such that $\support{(1+\lambda)x}{C}\leq 1$ and $\support{(1-\lambda)x}{C}\leq 1$. So $(1-\lambda)x$ and $(1+\lambda)x$ are both in $\cp$ which  contradicts to the fact that $x$ is an extreme point of $\cp$.

When $C$ contains the origin in its interior, $\cp$ is compact and $\rccp$ contains only zero vector. Hence,   $D_2$ is empty. If $C$ contains the origin,  for every $x$ in $\domsupport{x}{C}$,  we have $\support{x}{C}\geq 0$. If the origin is on its boundary, there exists a supporting hyperplane of $C$ through the origin. So $\cp$ is not compact and $\rccp$ contains a nonzero vector. Combined with the fact that $\support{x}{C}=0$  for all $x$ in $\CC$, $D_2$ can represent all the extreme rays of $\rccp$.

 It is clear that $C_3$ is empty if and only if $\support{y}{C}\geq 0$ for all $y\in\R{n}$. By Theorem \ref{charc}, this is equal to say that $C$ contains the origin.
Therefore, $D_3$ is empty if and only if $C$ contains the origin. Similar arguments can be used to show $\support{x}{C}=-1$ for every vector $x$ in $D_3$.

For every $x \in  \cone{\cp}$, there exists $\lambda\geq 0$ and $y\in \cp$ such that $x=\lambda y$. So $\support{x}{C}=\lambda \support{y}{C}<\infty$ and $x\in\domsupport{x}{C}$. On the other hand, for each $x\in\domsupport{x}{C}$,  if $\support{x}{C}=M>0$, then $x/M$ is in $\cp$ and   $x$ is in $\cone{\cp}$. Hence $\cone{\cp}=\domsupport{x}{C}$.
\end{proof}

\begin{remark}
 When $C$ does not contain the origin, it is not easy to identify whether the set  $D_2$ is empty. 
   The convex set $C$ in  Example \ref{nonequal} does not contain the origin,  $D_2=\{(-1, 0), (0, -1)\}$. 
    However, for the convex set  $C$ defined by $C=\{(x,y)\mid\  y\geq x+1, y\geq -x+1 \}$,   we have $\cp =\rccp = \{(x,y)\mid x+y\leq 0, \ y-x\leq 0\}$. The extreme rays of $\rccp$ are  $l_1=(-1,-1)$ and $l_2=(1, -1)$.   We have  $\support{l_1}{C}=\support{l_2}{C}=-1<0$. Hence,  the set  $D_2$ is empty.
\end{remark}

  When $C$ contains the origin in its interior, by Theorem \ref{small lemma},
 $D_2$ and $D_3$ are empty and  $C$ can be characterized by   $D_1$ alone.
However, when $C$ does not contain the origin in its interior, as shown by the following example, the linear functions with coefficients  in  $D_1$ or $D_1 \cup D_2$ can not characterize $C$ completely.

\paragraph{\bf Example \ref{nonequal} (continued)}

  In this example,   every linear function $f(x)=\langle l_1,x\rangle$ where $l_1\in D_1$  has maximal value $1$ on $C$, therefore,
\[E_1=\{(x,y)\mid c_1x+c_2y\leq 1,\ (c_1,c_2)\in D_1\}=\{(x,y)\mid -x+y\leq 1,\ x-y\leq 1,\ x+y\leq 3\}.\]
 The linear function $f(x)=\langle l_2,x\rangle$ where $l_2\in D_2$ has maximal value $0$ on $C$ and
 \[E_2=\{(x,y)\mid c_1x+c_2y\leq 0,\ (c_1,c_2)\in D_2\}=\{(x,y)\mid x\geq 0,\ y\geq 0\}.\]
 The linear function $f(x)=\langle l_3,x\rangle$ where $l_3\in D_3$ has maximal value $-1$ on $C$, hence,
  \[E_3=\{(x,y)\mid c_1x+c_2y\leq -1,\ (c_1,c_2)\in D_3\}=\{(x,y)\mid x+y\geq 1\}.\]
 \begin{figure}[!htb]
\includegraphics[width=1\textwidth]{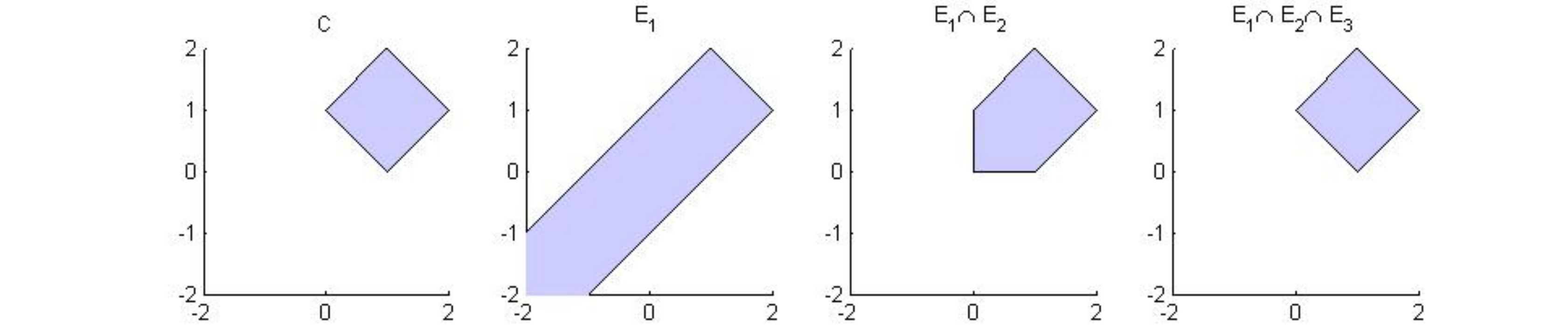}
\caption{Example \ref{nonequal} }
\end{figure}

We now  show that  a full dimensional closed convex set $C$ can be characterized completely by  elements  in $D_1, D_2$ and $D_3$.

\begin{theorem}\label{identify points in C}
Let $C\subset\R{n}$ be a full dimensional closed convex set. Then we have
\begin{equation}\label{condd1d2d3}
x\in C \iff  \left\{\begin{array}{ll}
\langle l_1,x\rangle   \leq 1  &  \forall \, l_1\in D_1,\\
\langle l_2,x\rangle  \leq 0  &\forall  \, l_2\in D_2,\\
\langle l_3,x\rangle   \leq -1 & \forall  \, l_3\in D_3.
  \end{array}
  \right.
  \end{equation}

\end{theorem}
\begin{proof}
Since $C$ is full dimensional, $\cp$ contains no lines and $D_1$, $D_2$ and $D_3$ are well defined. The necessity is clear.   Suppose on the other hand that $x$ satisfies the conditions on  the right hand side of (\ref{condd1d2d3}), we shall show that  $x \in C$.   By Theorem \ref{charc}, it is enough to show that $\langle l,x\rangle  \leq \support{x}{C}$ for every  $l\in \domsupport{x}{C}$. Let  $l_0=\support{l}{C}$, we prove that $\langle l,x\rangle\leq l_0$ in all  three cases below:
\begin{itemize}
\item  If $l_0>0$, then we have  $l/l_0\in \CA$.
 By Remark \ref{D123}, there exist $\lambda_i^1\geq 0,\  \lambda_j^2\geq 0,\  \lambda_k^3\geq 0$ and $x_i\in D_1$, $y_j\in D_{2}$, $z_k\in D_{32}$ satisfying the following equality:
\begin{equation*}
l/l_0=\sum\limits_{i}\lambda_i^1x_i+\sum\limits_{j}\lambda_j^2y_j+\sum\limits_{k}\lambda_k^3z_k, \ \sum\limits_{i}\lambda_i^1=1.
\end{equation*}
 According to the definitions of $D_1, D_2, D_3$,  we have
 \[\langle l/l_0,x\rangle= \sum\limits_{i}\lambda_i^1\langle x_i,x\rangle+\sum\limits_{j}\lambda_j^2\langle y_j,x\rangle+\sum\limits_{k}\lambda_k^3\langle z_k,x\rangle\leq \sum\limits_{i}\lambda_i^1=1.\]
\item
If $l_0=0$, then $l\in \CC$.  By Remark \ref{D123}, there exist $ \lambda_j^2\geq 0,\  \lambda_k^3\geq 0$ and $y_j\in D_{2}$, $z_k\in D_{32}$ satisfying the following equality:
\begin{equation*}
l=\sum\limits_{j}\lambda_j^2y_j+\sum\limits_{k}\lambda_k^3z_k.
\end{equation*}
So $\langle l,x\rangle=\sum\limits_{j}\lambda_j^2\langle y_j,x\rangle+\sum\limits_{k}\lambda_k^3\langle z_k,x\rangle\leq 0.$

\item
If $l_0<0$, then  $l/|l_0|\in C_3$. By Remark \ref{D123},  there exist $\lambda_i^1\geq 0,\  \lambda_j^2\geq 0, \ \lambda_k^3\geq 0$ and $x_i\in D_3$, $y_j\in D_{2}$, $z_k\in D_{32}$ satisfying the following equality:
\begin{equation*}
l/|l_0|=\sum\limits_{i}\lambda_i^1x_i+\sum\limits_{j}\lambda_j^2y_j+\sum\limits_{k}\lambda_k^3z_k, \ \sum\limits_{i}\lambda_i^1=1.
\end{equation*}
So $\langle l/|l_0|,x\rangle=\sum\limits_{i}\lambda_i^1\langle x_i,x\rangle+\sum\limits_{j}\lambda_j^2\langle y_j,x\rangle+\sum\limits_{k}\lambda_k^3\langle z_k,x\rangle\leq \sum\limits_{i}-\lambda_i^1=-1.$
\end{itemize}
\end{proof}

\begin{theorem}\label{identify one extreme}
Suppose $C\subset\R{n}$ is a full dimensional closed convex set. If there exists $x\in \R{n}$ such that $1-\langle l_1,x\rangle=0,\  \forall l_1\in D_1$, $-\langle l_2,x\rangle=0,\ \forall l_2\in D_2$, $-1-\langle l_3,x\rangle=0,\ \forall l_3\in D_3$, then $x$ is the only  extreme point of $C$ and $C$ is a translated convex cone.
\end{theorem}
\begin{proof}
By Theorem \ref{identify points in C}, $x$ is in $C$. By Theorem \ref{small lemma} and Remark \ref{D123}, for  every $l \in\domsupport{x}{C}$,
 there exist $\lambda_i^1\geq0,\  i=1,\ldots,i_1$, $\lambda_j^2\geq0,\  j=1,\ldots,j_2$ and $\lambda_k^3\geq 0,\  k=1,\ldots,k_3$ such that $l=\sum\limits_{i=1}^{i_1}\lambda_i^1 x_i+\sum\limits_{j=1}^{j_2}\lambda_j^2 y_j+\sum\limits_{k=1}^{k_3}\lambda_k^3 z_k$ for $x_i\in D_1$,  $y_j \in D_2$ and $z_k \in D_{32}$. Since for every  $y_j \in D_2$, $\support{y_j}{C}= 0$,
   we have the following inequality: 
\begin{align*}
\support{l}{C}&\leq \sum\limits_{i=1}^{i_1}\lambda_i^1\support{x_i}{C}+\sum\limits_{k=1}^{k_3}\lambda_k^3\support{z_k}{C}=\sum\limits_{i=1}^{i_1}\lambda_i^1-\sum\limits_{k=1}^{k_3}\lambda_k^3\\
&=\sum\limits_{i=1}^{i_1}\lambda_i^1\langle x_i, x\rangle+\sum\limits_{k=1}^{k_3}\lambda_k^3\langle z_k,x\rangle=\langle l,x\rangle.
\end{align*}

Furthermore,  it is clear that $\support{l}{C}\leq \langle l,x\rangle +\indicator{l}{\domsupport{l}{C}}$ for every $l\in\R{n}$. Take closure for both sides, we have:
\begin{align*}
\support{l}{C}&\leq \langle l,x\rangle +\indicator{l}{\cl{\domsupport{x}{C}}}\\
&=\langle l,x\rangle +\indicator{l}{(0^+C)^o}\ (\text{by Corollary \ref{dualpolar}})\\
&=\support{l}{x+\rpc}.
\end{align*}
This implies that $C\subseteq x+\rpc$. On the other hand, since $x\in C$,  we have $x+\rpc\subseteq C$. Therefore,   $C=x+\rpc$, i.e.  $C$  is a translated convex cone and contains $x$ as the only extreme point.
\end{proof}

\subsubsection{$C$ contains no lines}\label{sec3.1.1}
When $C\subset\R{n}$ is a full dimensional closed convex set and contains no lines, we set the slack operator $S_C$ to be

\begin{equation}\label{slackopadd}
S_C=\left\{ \begin{array}{ll}
\svpcone(x,y)=1-\langle x,y\rangle  &~\text{for}~ (x,y)\in C\times D_1, \\
\svpctwo(x,y)=-\langle x,y\rangle  &~\text{for}~ (x,y)\in C\times D_2,\\
\svpcthree(x,y)=-1-\langle x,y\rangle  &~\text{for}~  (x,y)\in C\times D_3.\\
\end{array}
 \right.
\end{equation}
In this definition, $D_1$, $D_2$ and $D_3$ are disjoint and   may be empty for some convex set $C$. If one of them is empty, we just remove the corresponding slack operator from the definition.
\begin{define}\label{KKfactor}
Let $K\subset\R{m}$ be a closed convex cone and $C\subset\R{n}$   a full dimensional  convex set containing no lines. We say that the slack operator $S_C$ defined by (\ref{slackopadd})
is $K$-factorizable, if there exist maps
\begin{align*}
A:C\rightarrow K, \ B_1: D_1\rightarrow \ks,\  B_2: D_2\rightarrow \ks,\  B_3: D_3\rightarrow \ks.
\end{align*}
such that
\begin{itemize}
\item $\svpci(x,y)=\langle A(x),B_i(y)\rangle$ for all $(x,y)\in C\times D_i$ and $i=1,2,3$.
\end{itemize}
\end{define}

\begin{theorem}\label{pointedcaseadd}
Let $K \subset\R{m}$ be a full dimensional  closed convex cone and  $C\subset\R{n}$  a full dimensional closed convex set containing no lines.   Assume   $C$ is not a translated cone.  If $C$ has a proper $K$-lift defined by (\ref{CKlift}), then the slack operator $S_C$  defined by (\ref{slackopadd}) 
is $K$-factorizable. Conversely, if $S_C$ defined by (\ref{slackopadd}) is 
$K$-factorizable, then $C$ has a $K$-lift defined by (\ref{CKlift}).
\end{theorem}

 The proof of Theorem \ref{pointedcase} can be modified slightly to show the correctness of Theorem \ref{pointedcaseadd}.

 \begin{example}\label{ex3.10}
   Consider $C=\{x\mid x\geq -1\}$. Let $K$ be $\mathcal{S}_+^3$ and
      \[L=\left\{\begin{pmatrix} a_{11} & a_{12} & a_{13} \\ a_{21} &  a_{22} & a_{23}\\ a_{31} &  a_{32} & a_{33}\end{pmatrix}\in \sm{3}\mid a_{11}=1,\ a_{13}=0,\ a_{23}=0,\ a_{33}=a_{12}+1\right\}.\] We construct a linear map
   $\pi$ from  $\sm{3}$ to $\R{1}$: $\begin{pmatrix} a_{11} & a_{12} & a_{13} \\ a_{21} &  a_{22} & a_{23}\\ a_{31} &  a_{32} & a_{33}\end{pmatrix}\rightarrow a_{33}$. It is easy to check that $C$ has a $K$-lift, i.e. $C=\pi(K\cap L)$.

   Now let us check whether  the slack operator $\svpc$ defined by (\ref{slackopadd}) is $K$-factorizable. Because $C$ contains the origin in its interior, according to Theorem \ref{small lemma}, $D_2$ and $D_3$ are empty. Since
  $\cp=\{x\mid -1\leq x\leq 0\}$, we have $D_1=\extc{\cp} \backslash 0=\{(-1)\}$.
   Let us  define the maps $A:C\rightarrow K, \ B_1: D_1\rightarrow \ks$ as
   \begin{eqnarray*}
   A(x)=
    \begin{pmatrix} 1 & x & 0 \\ x &  x^2 & 0\\ 0 &  0 & x+1\end{pmatrix}, ~
    B_1(y)=\begin{pmatrix} 0 & 0 & 0 \\ 0 &  0 & 0\\ 0 &  0 & -y\end{pmatrix}.
    \end{eqnarray*}
      Since $1-\langle x,y\rangle=\langle A(x),B_1(y)\rangle$ for all $(x,y)\in C\times D_1$, we claim that the slack operator $\svpc$ is $K$-factorizable.
 \end{example}
\begin{remark}\label{introduction}
 Although Definition \ref{KKfactor} and Theorem \ref{pointedcaseadd} have  extended the argument in \cite{Gouveia2013} to more general convex sets, it is not easy  to use. The main reason is that we have to define the map $A$ and  check whether $S_C$ is factorizable for all points in $C$. This is difficult  since $C$  usually contains infinite number of points  even when it is a polyhedron.
\end{remark}

 By   Theorem \ref{noncomexp},  if $C$ contains no lines, every point in it can be expressed as the convex combination of extreme points $\extc{C}$  and extreme directions $\extrc{0^+C}$. By introducing the cone lift of the recession cone $0^+C$ and defining its slack operator and cone factorization,  we  extend results in \cite{Gouveia2013} to  non-compact convex sets.

\begin{define}
Let $K\subset \R{m}$ be a closed convex cone. A  $K$-lift of a non-compact closed convex set  $C\subset\R{n}$ is a set $Q=K \cap L$ such that
\begin{equation}\label{Klift}
C=\pi(K\cap L), ~ 0^+C=\pi (K\cap 0^+L)
 \end{equation}
 where $L\subset\R{m}$ is an affine subspace and $\pi:\R{m}\rightarrow\R{n}$ is a linear map. We say that $Q$ is a proper $K$-lift of $C$, if $L\cap \inte{K}\neq \emptyset$.
\end{define}

We would like to emphasize that the condition $0^+C=\pi (K\cap 0^+L)$ is not redundant and can not be deduced from the condition $C=\pi(K\cap L)$ in general, see the following example.

\paragraph{\bf Example \ref{ex3.10} (continued)}

 Although we have  $C=\pi(K\cap L)$,  it is clear that
 \[\mathbb{R}_+^1=0^+C \neq \pi (K\cap 0^+L)\]
  since
    \[0^+L=\left\{\begin{pmatrix} a_{11} & a_{12} & a_{13} \\ a_{21} &  a_{22} & a_{23}\\ a_{31} &  a_{32} & a_{33}\end{pmatrix}\in \sm{3}\mid a_{11}=0,\ a_{13}=0,\ a_{23}=0,\ a_{33}=a_{12}\right\}\]
 and $~\pi (K\cap 0^+L)=\{0\}$.

We define the  slack operator $S_C$ of a full dimensional closed convex set $C$  as

\begin{equation}\label{slackop}
S_C=\left\{ \begin{array}{ll}
\svpcone(x,y)=1-\langle x,y\rangle  &~\text{for}~ (x,y)\in \extc{C}\times D_1, \\
\svpctwo(x,y)=-\langle x,y\rangle  &~\text{for}~ (x,y)\in \extc{C}\times D_2,\\
\svpcthree(x,y)=-1-\langle x,y\rangle  &~\text{for}~  (x,y)\in \extc{C}\times D_3,\\
 \svprci(x,y)=-\langle x,y\rangle  &~\text{for}~ (x,y)\in \extrc{\rpc}\times D_i, i=1,2,3.
 \end{array}
 \right.
 \end{equation}
In this definition, $D_1$, $D_2$ and $D_3$ are disjoint and   may be empty for some convex set $C$. If one of them is empty, we just remove the corresponding slack operator from the definition.

\begin{define}\label{Kfactor}
Let $K\subset\R{m}$ be a closed convex cone and $C\subset\R{n}$   a full dimensional closed convex set containing no lines. We say that the  slack operator $S_C$ defined by (\ref{slackop})
is $K$-factorizable, if there exist maps
\begin{align*}
&A_1:\extc{C}\rightarrow K, \  A_2:\extrc{\rpc}\rightarrow K,\\
&B_1: D_1\rightarrow \ks,\  B_2: D_2\rightarrow \ks,\  B_3: D_3\rightarrow \ks.
\end{align*}
such that
\begin{itemize}
\item $\svpci(x,y)=\langle A_1(x),B_i(y)\rangle$ for all $(x,y)\in \extc{C}\times D_i$ and $i=1,2,3$.
\item $\svprci(x,y)=\langle A_2(x),B_i(y)\rangle$ for all $(x,y)\in \extrc{\rpc}\times D_i$ and $i=1,2,3$.%
\end{itemize}
\end{define}

\begin{theorem}\label{pointedcase}
Let $K\subset\R{m}$ be a full dimensional  closed convex cone. Assume $C\subset\R{n}$ is a full dimensional closed convex set containing no lines and $C$ is not a translated cone.  If $C$ has a proper $K$-lift defined by (\ref{Klift}), then the slack operator $S_C$  defined by (\ref{slackop}) 
is $K$-factorizable. Conversely, if $S_C$ defined by (\ref{slackop}) is 
$K$-factorizable, then $C$ has a $K$-lift defined by (\ref{Klift}).
\end{theorem}

\begin{proof}
 Suppose $C$ has a proper $K$-lift, then we  set  $L=w_0+\lin$ in $\R{m}$ where $\lin$ is a linear subspace, $w_0 \in \inte{K}$ and $\pi:\R{m}\rightarrow\R{n}$ is a linear map such that $C=\pi(K\cap L), ~ 0^+C=\pi (K\cap 0^+L)$. Since  $ 0^+L=\lin$, we have $0^+C=\pi(K\cap \lin)$. We need to construct maps $A_1, A_2$ and $B_1, B_2, B_3$ from the $K$-lift that factorize the slack operator $S_C$.

  For every point  $x_1 \in \extc{C}$, there exists a point $w_1$  in the convex set $K \cap L$ such that $\pi(w_1)=x_1$. We define  $A_1(x_1):=w_1$. Moreover,   for every point $x_2\in \extrc{\rpc}$, there exists a point $w_2$ in the convex set $K\cap \lin$ such that $\pi(w_2)=x_2$. We define $A_2(x_2):=w_2$.


 The definitions of $B_1,B_2$ and $B_3$  are  similar to those given in
  \cite[Theorem 1]{Gouveia2013}, which use the properness condition to guarantee the strong duality holds.  The only difference is that for $l_1\in D_1$, $\max\{\langle l_1,x\rangle\mid x\in C\}$ is $1$,  for $l_2\in D_2$, $\max\{\langle l_2,x\rangle\mid x\in C\}$ is $0$ and  for $l_3\in D_3$, $\max\{\langle l_3,x\rangle\mid x\in C\}$ is $-1$.  Therefore, we only give the definitions and omit all proofs. For every $y_1 \in D_1$, we define $B_1(y_1):=z-\pi^*(y_1)$ where $z$ is any point in $\lin^\bot \cap (K^*+\pi^*(y_1))$ that satisfies $\langle w_0, z\rangle=1$. For every  $y_2 \in D_2$, we define $B_2(y_2):=z-\pi^*(y_2)$ where $z$ is any point in $\lin^\bot \cap (K^*+\pi^*(y_2))$ that satisfies $\langle w_0, z\rangle=0$.   For every  $y_3 \in D_3$, we define $B_3(y_3):=z-\pi^*(y_3)$ where $z$ is any point in $\lin^\bot \cap (K^*+\pi^*(y_3))$ that satisfies $\langle w_0, z\rangle=-1$.  It remains to check that $\svpci$ and $\svprci$ have  $K$-factorizations given in Definition \ref{Kfactor}. The $K$-factorization of $\svpci$ can be checked by the same method used in \cite[Theorem 1]{Gouveia2013}.
  For  each $x_2  \in \extrc{\rpc}$ and $y_i\in D_i,\ i=1,2,3$, we have
\begin{align*}
\langle x_2,y_i\rangle&=\langle \pi(w_2), y_i\rangle=\langle w_2, \pi^*(y_i)\rangle=\langle w_2, z-B_i(y_i)\rangle\\
&=-\langle w_2,B_i(y_i)\rangle=-\langle A_2(x_2),B_i(y_i)\rangle.
\end{align*}
Therefore, $\svprci$ is $K$-factorizable according to Definition \ref{Kfactor}.

 Suppose on the other hand that  $S_C$ 
 is $K$-factorizable, i.e. there exist maps $A_1, A_2$ and $B_1, B_2,B_3$ such that
   $\svpci(x,y)=\langle A_1(x),B_i(y)\rangle$ for all $(x,y)\in \extc{C}\times D_i$, $i=1,2,3$
 and $\svprci(x,y)=\langle A_2(x),B_i(y)\rangle$ for all $(x,y)\in \extrc{\rpc}\times D_i$, $i=1,2,3$.
We  construct the affine space
\begin{align*}
L=&\{(x,z) \in \R{n} \times \R{m} \mid 1-\langle x,y_1\rangle=\langle z,B_1(y_1)\rangle,\ \forall y_1\in D_1,\\
& \hfill -\langle x,y_2\rangle=\langle z,B_2(y_2)\rangle,\ \forall y_2\in D_2,
\ -1-\langle x,y_3\rangle=\langle z,B_3(y_3)\rangle,\ \forall y_3\in D_3\},
\end{align*}
and  let $L_K$ be the projection of $L$ onto the second component $z$.

  We need to show firstly  that $0\notin L_K$. If $0\in L_K$, there exists $x\in \R{n}$ such that $1-\langle x,y_1\rangle=0,\  \forall y_1\in D_1$, $-\langle x,y_2\rangle=0,\ \forall y_2\in D_2$, $-1-\langle x,y_3\rangle=0,\ \forall y_3\in D_3$.
By Theorem \ref{identify one extreme}, $C$ is a translated cone and  this contradicts to the assumption. For  each $x \in \extc{C}$, we have $A_1(x) \in K \cap L_K$, then $K \cap L_K \neq \emptyset$.

For every $x\in \R{n}$, if there exists $z\in K$ such that $(x,z) \in L$, then $\langle x,y_1\rangle=1-\langle z,B_1(y_1)\rangle \leq 1,\ \forall y_1\in D_1$,
$\langle x,y_2\rangle=-\langle z,B_2(y_2)\rangle \leq 0, \ \forall y_2\in D_2$, and $\langle x,y_3\rangle= -1-\langle z,B_3(y_3)\rangle\leq -1,\ \forall y_3\in D_3$. By Theorem \ref{identify points in C}, we have  $x\in C$. Hence,  $\pi(K\cap L_K) \subseteq C$.

 Since $C$ contains no lines,  we can show that for every $z\in K\cap L_K$, there exists unique $x_z\in\R{n}$ such that $(x_z,z)\in L$.  Hence, the map from $z$ to $x_z$ is a well defined affine map. Since the origin is not in $L_K$,  we can extend the map to a linear map: $\R{m}\rightarrow \R{n}$. In order to prove that $C=\pi(K\cap L_K)$,  we only  need to show that  $C\subseteq \pi(K\cap L_K)$.

For every $x\in C$, there exist $\lambda_i^1\geq 0,\  i=1,\ldots, i_1$, $\lambda_j^2\geq 0,\  j=1,\ldots, j_2$ such that
\begin{equation*}
x=\sum\limits_{i=1}^{i_1}\lambda_i^1x_i+\sum\limits_{j=1}^{j_2}\lambda_j^2y_j,\ \sum\limits_{i=1}^{i_1}\lambda_i^1=1
\end{equation*}
where $x_i \in \extc{C}$ and $y_j \in \extrc{\rpc}$.
Let $z=\sum\limits_{i=1}^{i_1}\lambda_i^1 A_1(x_i)+\sum\limits_{j=1}^{j_2}\lambda_j^2A_2(y_j)$. Since $S_C$ is  $K$-factorizable,  it is easy to check that  $z \in K\cap L_K$ and therefore, $x=\pi(z)\in \pi(K\cap L_K)$. We can deduce that $C = \pi(K\cap L_K)$.

Furthermore,  we need to show that $\rpc=\pi(K\cap 0^+L_K)$. Since $C=\pi(K\cap L_K)$, we know  that $\rpc \supseteq \pi(K\cap 0^+L_K)$. On the other hand,   for every $x\in \extrc{\rpc}$, by the definition of $L$, we claim that $A_2(x)$ is in $K\cap 0^+L_K$. Therefore, we  have  $\rpc=\pi(K\cap 0^+L_K)$.
\end{proof}


The following example shows that the $K$-factorization of $\svpci(x,y)$ for  $(x,y)\in \extc{C}\times D_i$ and $i=1,2,3$
 can not guarantee that the convex set $C$ has a $K$-lift defined by (\ref{Klift}). 
 It is necessary to consider the $K$-factorization of $\svprci(x,y)$ for all $(x,y)\in \extrc{\rpc}\times D_i$ and $i=1,2,3$ too.

\begin{example}\label{example3.14}
Let $C$ be a 3-dimensional polyhedron in $\R{3}$ defined by the following inequality:
\begin{equation*}
C=\left\{ (x_1,x_2,x_3)\in\mathbb{R}^3:
\left( \begin{array}{ccc}
\noalign{\medskip}1&\frac{\sqrt {3}}{3}&0\\
\noalign{\medskip}0&\frac{2\sqrt {3}}{3}&0\\
\noalign{\medskip}-1&\frac{\sqrt {3}}{3}&0\\
\noalign{\medskip}-1&-\frac{\sqrt {3}}{3}&0\\
\noalign{\medskip}0&-\frac{2\sqrt {3}}{3}&0\\
\noalign{\medskip}1&-\frac{\sqrt {3}}{3}&0\\
\noalign{\medskip}0&0&-1
\end{array}
\right)
\left(
\begin{array}{c}
x_1\\
x_2\\
x_3
\end{array}
\right)
\leq
\left(
  \begin{array}{c}
    1\\
    1\\
    1\\
    1\\
    1\\
    1\\
    0
  \end{array}
\right)
\right\}.
\end{equation*}
\begin{figure}
\includegraphics[width=0.6\textwidth]{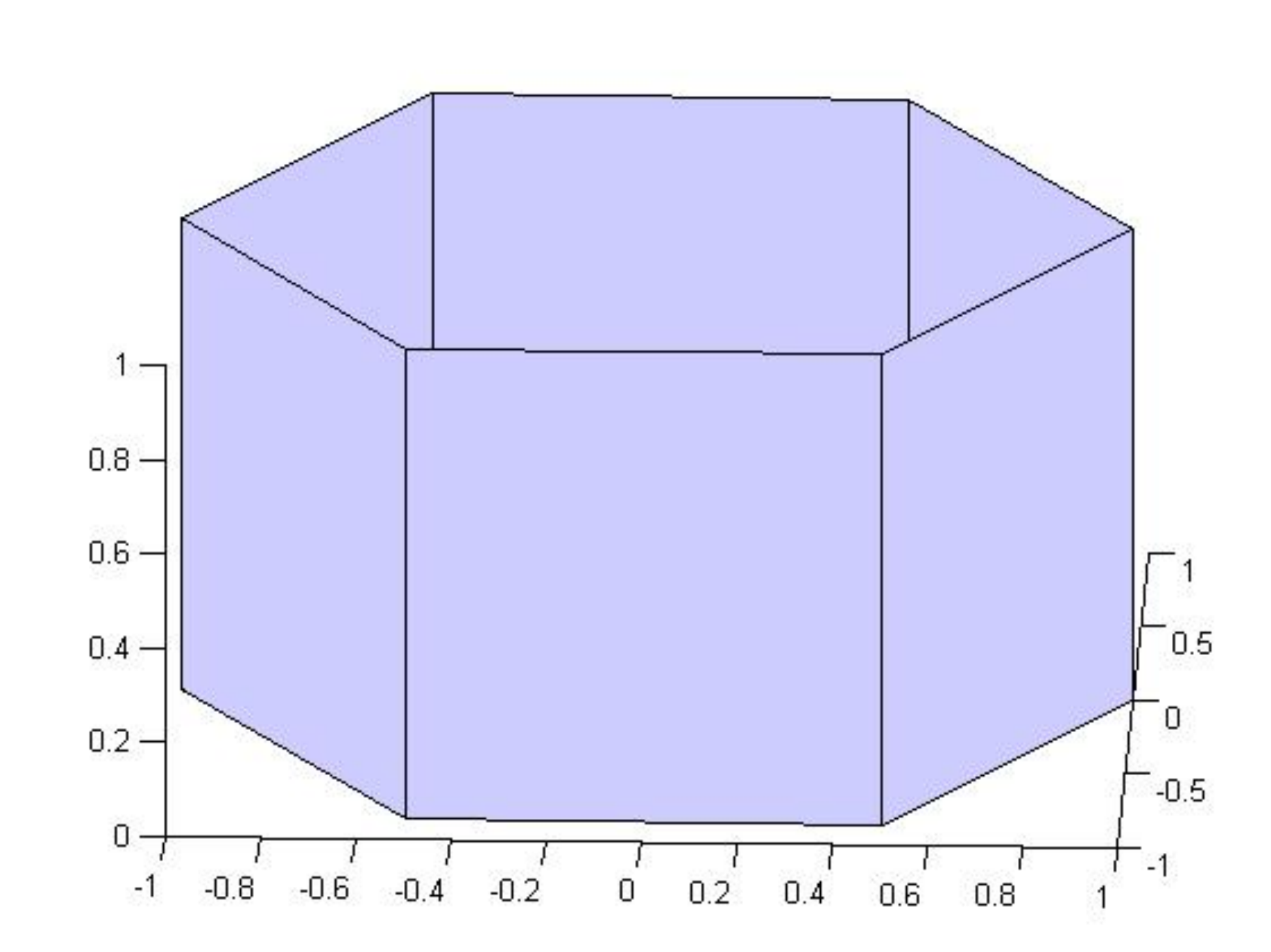}
\caption{ Example \ref{example3.14}}\label{fig3}
\end{figure}
The six vertices of $C$  are $\{(cos(i\pi/3),sin(i\pi/3),0),\ i=0,\ldots,5\}$ and $0^+C=\cone{\{(0,0,1)\}}$. According to Definition \ref{slacek},  its slack matrix   is
\begin{equation*}
S:=
\left (
\begin{array}{ccccccc}
0&0&1&2&2&1&0\\
1&0&0&1&2&2&0\\
2&1&0&0&1&2&0\\
2&2&1&0&0&1&0\\
1&2&2&1&0&0&0\\
0&1&2&2&1&0&0\\
0&0&0&0&0&0&1
\end{array}
\right).
\end{equation*}
 It has been shown in \cite[Example 2]{Gouveia2013} that the first $6 \times 6$ submatrix $S_H$ of  $S$  has  a $\RR_{+}^5$-factorization. However, we claim that the matrix $S$  does not have a $\RR_+^5$-factorization. If it does, we can assume it has the following nonnegative decomposition:
\begin{equation*}
\left(
\begin{array}{cc}
S_H&0\\
0&1
\end{array}
\right)
=
\left(
\begin{array}{cc}
A_{11}&A_{12}\\
A_{21}&A_{22}
\end{array}
\right)
\left(
\begin{array}{cc}
B_{11}&B_{12}\\
B_{21}&B_{22}
\end{array}
\right)
\end{equation*}
Since $A_{21}B_{11}+A_{22}B_{21}=0$, we have $A_{22}B_{21}=0$. We claim that  $A_{22}=0$. Otherwise, $B_{21}$ will be zero and $S_H=A_{11}B_{11}$. This contradicts to the fact that $S_H$ has no $\RR_+^4$-factorization. From $A_{21}B_{12}+A_{22}B_{22}=1$ and $A_{22}=0$, we have $B_{12}\neq 0$.  However, since $A_{11}B_{12}+A_{12}B_{22}=0$, there exists  one column of $A_{11}$ which is a zero column. Therefore, $S_H$ has a nonnegative decomposition in $\RR_{+}^4$ which is also a contradiction. Hence, according to Theorem \ref{pointedcase}, $C$ has no $\RR_{+}^5$-lift.
\end{example}

In Theorem \ref{pointedcase}, we have assumed that the full dimensional closed convex set $C$ is not a translated cone. If $C$ is a translated cone, a $K$-lift of $C$ can be defined as
\begin{equation}\label{conelift}
C=b+\pi(K\cap L),
\end{equation}
where $b\in \R{n}$ is a constant vector, $L$ is a linear space and $\pi:\R{m}\rightarrow\R{n}$ is a linear map.  In this case, we only need to  characterize $\extrc{\rpc}$. Without loss of generality, we can assume $b=0$ and $C$ is a cone.
We define  the slack operator $S_C$ as
\begin{equation}\label{conefactorizable}
S_C(x,y)=-\langle x,y\rangle  ~\text{for}~ (x,y)\in \extrc{C}\times \extrc{\cp}. 
 \end{equation}
We say that the slack operator $S_C$  is  $K$-factorizable, if there exist maps
\begin{equation*}
A: \extrc{C} \rightarrow K,  ~B: \extrc{\cp}\rightarrow \ks
\end{equation*}
such that
\begin{itemize}
\item $S_C(x,y)=\langle A(x),B(y)\rangle$ for all $(x,y)\in \extrc{C}\times \extrc{\cp}$.
\end{itemize}

\begin{theorem}\label{pointedcone}
Let $K$ be a full dimensional convex cone in $\R{m}$ and $C$  a full dimensional closed pointed convex cone in $\R{n}$.  
If $C$ has a proper $K$-lift defined by (\ref{conelift}), then $S_C$ defined by (\ref{conefactorizable}) is  $K$-factorizable. Conversely, if $S_C$ defined by (\ref{conefactorizable}) is $K$-factorizable, then $C$ has a $K$-lift defined by (\ref{conelift}).
\end{theorem}

The proof that $S_C$ is $K$-factorizable if $C$ has a proper $K$-lift is similar to the one given for Theorem \ref{pointedcase}. We omit the proof here. Below we give a short proof to show that   $C$ has a $K$-lift if $S_C$ is $K$-factorizable.
\begin{proof}
Suppose $S_C$ is $K$-factorizable, we construct the linear space
\begin{align*}
L=\{(x,z)\mid -\langle x,y\rangle=\langle z,B(y)\rangle,\ \forall y\in \extrc{\cp}\}.
\end{align*}

Let $L_K$ be the projection of $L$ onto the second component $z$. Since $C$ contains no lines, for every $z\in L_K$, there exists  unique $x_z\in\R{n}$ such that $(x_z,z)\in L$. Hence, we can define a linear map  $\pi$: $L_K\rightarrow x_z$. Since $L$ is a linear space, we can extend $\pi$ to a linear map: $\R{m}\rightarrow\R{n}$.

For every $x \in \RR^n$, if there exists $z \in K$ such that $(x,z)$ is in $L$,
then $\langle x,y\rangle\leq 0$ for all $y\in \extrc{\cp}$. Hence  $x\in C^{oo}=\cl{C}=C$ since $C$ is a closed convex set.  We have $\pi(K\cap L) \subseteq C$. On the other hand, since $S_C$ is $K$-factorizable, for every $x \in \extrc{C}$, $(x, A(x)) \in L$,  hence  $C \subseteq \pi(K\cap L)$. The proof is completed.
\end{proof}

%

\subsubsection{$C$ contains  lines}\label{containline}

When $C$ is a full dimensional closed convex set containing  lines, Definition \ref{KKfactor} and Theorem \ref{pointedcaseadd} can be generalized without any change.  However, when $C$ contains lines, it has no extreme points  and $\rpc$ contains no extreme rays,   Definition \ref{Kfactor} and Theorem \ref{pointedcase} need to be  adjusted properly.

Let $\linc$  denote the lineality space of $C$ and  $\{l_1\ldots,l_s\}$ be an orthogonal basis of $\linc$.  The convex set
containing lines can be decomposed as
\begin{equation}\label{decomposelinear}
C=\pointed+\linc,
\end{equation}
where   $\pointed=C\cap \linco$ is a closed convex set containing  no lines and $\linco$ is the orthogonal complement of $\linc$.

\begin{lemma}\label{affinehull}
$\linco$ is  the affine hull of $\cp$.
\end{lemma}
\begin{proof}
Since $C=\pointed+\linc$, we have $\cp=C_0^o\cap \linco$. The convex set  $C_0$ contains no lines,  then $0^+C_0$ contains no lines. By Corollary \ref{dualpolar}  and Theorem \ref{small lemma}, we claim $\cl{\cone{C_0^o}}=(0^+C_0)^o$. Since $(0^+C_0)^o$ contains an interior,   it is clear that  $\cone{C_0^o}$ contains an  interior. As $C_0$ is in $\linco$, $\cone{C_0^o}$ contains $\linc$. Hence  $\cone{\cp}=\cone{C_0^o}\cap \linco$ has an interior in $\linco$. Furthermore,  $\cp=C_0^o\cap \linco$ has an interior in $\linco$ too. Hence  $\linco$ is the affine hull of $\cp$.
\end{proof}

We define  the slack operator $S_C$ of a full dimensional closed convex set $C$ containing lines as
\begin{equation}\label{slackop2}
S_C=\left\{ \begin{array}{ll}
\svcone(x,y)=1-\langle x,y\rangle  &~\text{for}~ (x,y)\in \extc{\pointed}\times D_1, \\
\svctwo(x,y)=-\langle x,y\rangle  &~\text{for}~ (x,y)\in \extc{\pointed}\times D_2,\\
\svcthree(x,y)=-1-\langle x,y\rangle  &~\text{for}~  (x,y)\in \extc{\pointed}\times D_3,\\
\svrci(x,y)=-\langle x,y\rangle  &~\text{for}~ (x,y)\in \extrc{\rc}\times D_i, i=1,2,3,\\
\svl(x,y)=\langle x,y\rangle  & ~\text{for}~ (x,y)\in  \{  l_1,\ldots,  l_s\}\times \{  l_1,\ldots,  l_s\}.
 \end{array}
 \right.
 \end{equation}

\begin{define}\label{Kfactorl}
Let $K\subset\R{m}$ be a full dimensional closed convex cone and $C\subset\R{n}$  a full dimensional non-compact closed convex set containing lines. We say that the slack operator $S_C$ defined by (\ref{slackop2})
is $K$-factorizable, if there exist maps
\begin{align*}
&A_1:\extc{\pointed}\rightarrow K, \  A_2:\extrc{\rc}\rightarrow K,\  A_3: \{  l_1,\ldots,  l_s\}\rightarrow K,\\
&B_1: D_1\rightarrow \ks,\  B_2: D_2\rightarrow \ks,\  B_3: D_3\rightarrow \ks,\  F:\{  l_1,\ldots,  l_s\}\rightarrow \R{m}.
\end{align*}
such that
\begin{itemize}
\item $\svci(x,y)=\langle A_1(x),B_i(y)\rangle$ for all $(x,y)\in \extc{\pointed}\times D_i$ and $i=1,2,3$,
\item $\svrci(x,y)=\langle A_2(x),B_i(y)\rangle$ for all $(x,y)\in \extrc{\rc}\times D_i$ and $i=1,2,3$,
\item $\svl(x,y)=\langle A_3(x),F(y)\rangle$ for all $(x,y)\in \{ l_1,\ldots, l_s\}\times \{l_1,\ldots,l_s\}$,
\item $\langle A_3(x),B_i(y)\rangle=0$ for all $(x,y)\in\{ l_1,\ldots,  l_s\} \times D_i$ and $i=1,2,3$,
\item $\langle A_1(x),F(y)\rangle=0$ for all $(x,y) \in \extc{\pointed}\times \{ l_1,\ldots, l_s\}$,
\item $\langle A_2(x),F(y)\rangle=0$ for all $(x,y) \in \extrc{\rc} \times \{ l_1,\ldots, l_s\}$.
\end{itemize}
\end{define}

\begin{theorem}\label{unpointedcase}
Let $K$ be a full dimensional closed convex cone in $\R{m}$. Assume $C$ is a full dimensional closed convex set in $\R{n}$ which can be decomposed as (\ref{decomposelinear}) and $\pointed$ is not a translated cone. If $C$ has a proper  $K$-lift defined by ($\ref{Klift}$), then the  slack operator $S_C$ defined by ($\ref{slackop2}$) is  $K$-factorizable. Conversely, if $S_C$ defined by ($\ref{slackop2}$) is  $K$-factorizable, then $C$ has a  $K$-lift defined by ($\ref{Klift}$).
\end{theorem}
\begin{proof}
Suppose $C$ has a proper $K$-lift, then we  set  $L=w_0+\lin$ in $\R{m}$ where $\lin$ is a linear subspace, $w_0 \in \inte{K}$ and $\pi:\R{m}\rightarrow\R{n}$ is a linear map such that $C=\pi(K\cap L), ~ 0^+C=\pi (K\cap 0^+L)$. Since  $ 0^+L=\lin$, we have $0^+C=\pi(K\cap \lin)$. We need to construct maps $A_1, A_2, A_3$, $B_1, B_2, B_3$ and $F$ that factorize the  slack operator $S_C$ from the $K$-lift.
  We can define $A_1,A_2,B_1,B_2,B_3$ by the same way used in the proof of Theorem \ref{pointedcase}. For every $l_i$, $i=1,\dots,s$, there exists a point $w_i\in K\cap\lin$  such that $\pi(w_i)=l_i$. So we define $A_3(l_i):=w_i$ for $i=1,\dots,s$.
Furthermore, we define $F(l_i):=\pi^*(l_i)$ for $i=1,\dots,s$.

The equalities for $\svci$, $\svrci$, $i=1,\ldots,3$ in Definition \ref{Kfactorl} can be checked by the
  same method used in the proof  of Theorem \ref{pointedcase}.
For  each $x,y\in\{ l_1,\ldots, l_s\}$, we  have
\begin{equation*}
\langle x,y\rangle=\langle \pi(A_3(x)),y\rangle=\langle A_3(x),F(y)\rangle.
\end{equation*}
For  each $x\in\{ l_1,\ldots, l_s\}$ and $y\in D_i$, we  have
\begin{equation*}
\langle A_3(x),B_i(y)\rangle=\langle A_3(x),z-\pi^*(y)\rangle=-\langle \pi(A_3(x)),y\rangle=-\langle x,y\rangle=0.
\end{equation*}
For each  $x\in \extc{\pointed}$ and  $y\in\{ l_1,\ldots, l_s\}$, we have
\begin{align*}
\langle A_1(x),F(y)\rangle=\langle \pi(A_1(x)), y\rangle=\langle x,y\rangle=0.
\end{align*}
For each  $x\in \extrc{\rc}$ and  $y\in\{ l_1,\ldots, l_s\}$, we have
\begin{align*}
\langle A_2(x),F(y)\rangle=\langle \pi(A_2(x)), y\rangle=\langle x,y\rangle=0.
\end{align*}
Therefore, $S_C$ is $K$-factorizable.

Suppose $S_C$ is $K$-factorizable.
 We need to construct an affine space $L$:
\begin{align*}
L=&\{(x,z)\in\R{n}\times \R{m}\mid x=x_1+x_2 \text{ such that } \ x_1\in \linco \text{ and } x_2\in \linc,\\
&1-\langle x_1,y_1\rangle=\langle z,B_1(y_1)\rangle,\ \forall y_1\in D_1,\  -\langle x_1,y_2\rangle=\langle z,B_2(y_2)\rangle,\ \forall y_2\in D_2,\\
&-1-\langle x_1,y_3\rangle=\langle z,B_3(y_3)\rangle,\ \forall y_3\in D_3,\  \langle x_2,l_i \rangle=\langle z, F(l_i)\rangle,\ \forall i=1,\ldots,s\}.
\end{align*}
Let $L_K$ be the projection of $L$ onto the second component $z$.

We need to show that $0\notin L_K$. If $0\in L_K$, there exists $x=x_1+x_2$ such that $1-\langle x_1,y_1\rangle=0,\  \forall y_1\in D_1$, $-\langle x_1,y_2\rangle=0,\ \forall y_2\in D_2$, $-1-\langle x_1,y_3\rangle=0,\ \forall y_3\in D_3$. By Theorem \ref{identify one extreme}, $C_0$ is a translated convex cone and this leads to a contradiction. Moreover, $K \cap L_K \neq \emptyset$ since for each $x \in \extc{C_0}$, we have $A_1(x) \in K \cap L_K$.

Now we  prove that for each $z\in K\cap L_K$, there exists  unique $x\in \R{n}$ such that $(x,z)\in L$. If for some $z\in K\cap L_K$, there exist two different points $x^1=x_1^1+x_2^1$ and $x^2=x_1^2+x_2^2$ such that both $(x^1,z)$ and $(x^2,z)$ are in $L$. Because $\langle x_1^1-x_1^2,y\rangle=0$ for $y\in D_1\cup D_2\cup D_3$, 
we have  $\langle x_1^1-x_1^2,y\rangle=0$ for $y\in \cp$. According to Lemma \ref{affinehull},
$\linco$ is  the affine hull of $\cp$, hence  $x_1^1-x_1^2\in\linc$. Since  $x_1^1-x_1^2$ is also in $\linco$,  we have $x_1^1-x_1^2=0$. Furthermore,  because  $\langle x_2^1-x_2^2,l_i\rangle =0$, $1\leq i\leq s$ and $x_2^1-x_2^2\in \linc$, we also have $x_2^1-x_2^2=0$.  Hence, the map from $z$ to $x_z$ is a well defined affine map. Since the origin is not in $L_K$,  we can extend it to a linear map: $\R{m}\rightarrow \R{n}$.
By the same method used in proving Theorem \ref{pointedcase}, we can show $C=\pi(K\cap L_K)$ and $\rpc=\pi(K\cap 0^+L_K)$.
\end{proof}

When $C$ is a closed translated convex cone that contains lines, $C$  can also be decomposed as (\ref{decomposelinear}) and we have results similar to those given in Theorem \ref{pointedcone}. The slack operator $S_C$ of a  convex set $C$ containing lines is defined as

\begin{equation}\label{conefactorizablel}
S_C=\left\{ \begin{array}{ll}
\svc(x,y)=-\langle x,y\rangle  &~\text{for}~ (x,y)\in \extrc{\pointed}\times \extrc{\cp}, \\
\svl(x,y)=-\langle x,y\rangle  &~\text{for}~ (x,y)\in \{l_1,\ldots, l_s\}\times \{l_1,\ldots,l_s\}.\\
 \end{array}
 \right.
 \end{equation}

\begin{define}
We say that the slack operator $S_C$ defined by (\ref{conefactorizablel}) is $K$-factorizable, if there exist maps
\begin{align*}
& A_2:\extrc{\pointed}\rightarrow K,\  A_3: \{  l_1,\ldots,  l_s\}\rightarrow K,\\
&B: \extrc{\cp} \rightarrow K^*, \  F:\{  l_1,\ldots, l_s\}\rightarrow \R{m}.
\end{align*}
such that
\begin{itemize}
\item $\svc(x,y)=\langle A_2(x),B(y)\rangle$ for all $(x,y)\in \extrc{\pointed}\times \extrc{\cp}$,
\item $\svl(x,y)=\langle A_3(x),F(y)\rangle$ for all $(x,y)\in \{l_1,\ldots, l_s\}\times \{l_1,\ldots,l_s\}$,
\item $\langle A_2(x),F(y)\rangle=0$ for all $(x,y) \in \extrc{\pointed} \times \{ l_1,\ldots, l_s\}$,
\item $\langle A_3(x),B(y)\rangle=0$ for all $\{ l_1,\ldots,  l_s\} \times \extrc{\cp}$.
\end{itemize}
\end{define}

\begin{theorem}\label{conecontline}
Let $K$ be a full dimensional convex cone in $\R{m}$ and $C$ is a full dimensional closed translated convex cone in $\R{n}$ that contains lines and $C$ can be decomposed as (\ref{decomposelinear}).  If $C$ has a proper $K$-lift defined by (\ref{conelift}), then $S_C$ defined by (\ref{conefactorizablel}) is  $K$-factorizable. Conversely, if $S_C$ defined by (\ref{conefactorizablel}) is $K$-factorizable, then $C$ has a $K$-lift defined by (\ref{conelift}).
\end{theorem}

Theorem \ref{conecontline} can be proved using similar arguments for  Theorem \ref{pointedcone} and Theorem \ref{unpointedcase}.

\subsection{$C$ is not full dimensional}
When $C$ is not a full dimensional convex set, it has a non-trivial affine hull.
\begin{theorem}\label{notfulldim}
Let $C$ be a closed convex set in $\R{n}$. The polar $C^o$ contains lines if and only if $C$ is contained in a non-trivial linear space. When $C$ contains the origin, $C$ is not full dimensional if and only if $C^o$ contains lines.
\end{theorem}
\begin{proof}
$C^o$ contains  lines if and only if there exists $a\in\R{n}$ such that $\support{a}{C}\leq 0$ and $\support{-a}{C}\leq 0$, i.e. $C$ is contained in the set $\{x\mid a^Tx=0\}$.

When $C$ contains the  origin, $C$ is not full dimensional if and only if there exists $a\in\R{n}$ such that $C$ is contained in $\{x\mid a^Tx=0\}$.
\end{proof}

We assume that the convex set $C$ is not full dimensional and contains no lines. Then $\cp$ may or may not  contain  lines. If $\cp$ contains no lines, we  have the same results as the case that $C$ is  full dimensional. When $\cp$ contains  lines, there exists no extreme point or extreme direction in $\cp$ and the sets $D_1$, $D_2$ and $D_3$ are empty.  Let $\lincc$  denote the lineality space of $\cp$. Assume  $\cp=C'+\lincc$ such that $C'=\cp \cap  \lincco$. The closed convex set $C'$  contains no lines. 
It is clear that   $0^+C'=\rccp\cap \lincco$. Recall that $C_3=\{x\mid \support{x}{C} \leq -1\}$, $C_3'=C_3\cap \lincco$  contains no lines. Let
\begin{equation*}
D_1'=\extc{C'} \backslash 0,\  D_2'=\extrc{0^+C'} \cap \{x\mid\support{x}{C'}=0\},\  D_3'=\extc{C_3'}.
\end{equation*}
Let
\begin{equation*}
D_{32}'=\extrc{0^+C'} \cap \{x\mid\support{x}{C'}=-1\}.
\end{equation*}
Then $D_{32}'\subseteq D_3'$.

 \begin{theorem}\label{identify points in nonfull C}
 Assume a closed convex set $C\subset\R{n}$ is not full dimensional and contains no lines. For a vector $x\in \R{n}$, $x\in C$ if and only if for every $l\in D_1'$, $\langle l,x\rangle\leq 1$, for every $l\in D_2'$, $\langle l,x\rangle\leq 0$, for every $l\in D_3'$, $\langle l,x\rangle\leq -1$ and $x \in \lincco$, where   $\lincc$  is the lineality space of $\cp$.
\end{theorem}
\begin{proof}
Similar to the proof of  Theorem \ref{identify points in C}.
\end{proof}

Assume the closed convex set $C$ is not full dimensional and contains no lines. By replacing $D_i$ by $D_i'$ for $i=1,2,3$ in  Theorem \ref{slackop} and Definition \ref{Kfactor}, we can define the  slack operator $S_C$ and its $K$-factorization, then all results in  Subsection \ref{sec3.1.1} can be extended trivially to the case that $C$ is not full dimensional. Although results in Subsection \ref{containline} can also be extended  to the case that the closed convex set $C$ is not full dimensional and contains  lines, it becomes much more complicated and we omit the discussions here.

\section{Cone lifts of polyhedra}\label{polyhedralift}

  Similar to  \cite[Section 3]{Gouveia2013}, we specialize  results given in previous section to the case of cone lifts of polyhedra. Let $C\subset\R{n}$ be a polyhedron defined by a set of  linear inequalities:
\begin{align}\label{convexrepresention}
C=\{x\in\R{n}: &\ f_1(x)\leq\alpha_1,\ldots,f_{k_1}(x)\leq\alpha_{k_1},\ g_1(x)\leq0,\ldots,g_{k_2}(x)\leq0,\\
\nonumber &\ h_1(x)\leq-{\beta_1},\ldots,h_{k_3}(x)\leq-{\beta_{k_3}}\},
\end{align}
where $ {\alpha_i}>0$ for $1 \leq i \leq k_1$ and ${\beta_j}>0$ for $ 1 \leq j \leq k_3$.
The  recession cone of $C$  has the following form:
\begin{align}
\rpc=\{x\in\R{n}:&f_1(x)\leq0,\ldots,f_{k_1}(x)\leq0,\ g_1(x)\leq0,\ldots,g_{k_2}(x)\leq0,\\
\nonumber &h_1(x)\leq0,\ldots,h_{k_3}(x)\leq0\}.
\end{align}
  Let the convex set $C$ be generated by  a set of points $c_1,\ldots,c_t$ and  directions $r_1,\ldots,r_s$. We extend the  definition of a slack matrix   in \cite{Gouveia2013,Yannakakis1991}.

 \begin{define}\label{slacek}
  We define the  \itshape{slack matrix} of $C$  as  $[S_1^T,  S_2^T, S_3^T]^T$,  where
\begin{enumerate}
\item  $S_1 \in \mathbb{R}^{k_1\times (t+s)}$ whose $(i,j)$-entry is $\alpha_i-f_i(c_j)$ for $i=1,\ldots,k_1$, $j=1,\ldots,t$ and $(i,t+j)$-entry is $-f_i(r_j)$ for $i=1,\ldots,k_1$, $j=1,\ldots,s$.

\item  $S_2 \in \mathbb{R}^{k_2\times (t+s)}$ whose $(i,j)$-entry is $-g_i(c_j)$ for $i=1,\ldots,k_2$, $j=1,\ldots,t$ and $(i,t+j)$-entry is $-g_i(r_j)$ for $i=1,\ldots,k_2$, $j=1,\ldots,s$.

\item  $S_3 \in \mathbb{R}^{k_3\times (t+s)}$ whose $(i,j)$-entry is $-\beta_i-h_i(c_j)$ for $i=1,\ldots,k_3$, $j=1,\ldots,t$ and $(i,t+j)$-entry is $-h_i(r_j)$ for $i=1,\ldots,k_3$, $j=1,\ldots,s$.

\end{enumerate}
 Assume $C$ is a full dimensional polyhedron containing no lines, the  slack matrix $S$ is called the  canonical slack matrix of $C$ if $f_i, g_i, h_i$ represent the facets of $C$, $\alpha_i=1,\beta_j=1$ for $i=1,\ldots,k_1$, $j=1,\ldots, k_3$ and $c_1,\ldots,c_t$ and $r_1,\ldots,r_s$ are the vertices and the extreme directions of $C$  respectively.
\end{define}

\begin{define}\label{kfactorizable}\cite[Definition 7]{Gouveia2013}
Let  $M=(M_{ij})\in\mathbb{R}_+^{p\times q}$ be a nonnegative matrix and $K$ a closed convex cone. Then, a $K$-factorization of $M$ is a pair of ordered sets $a_1,\ldots,a_p\in K$ and $b_1,\ldots,b_q\in K^*$ such that $\langle a_i,b_j\rangle=M_{ij}$.
\end{define}

Definition \ref{kfactorizable} generalizes nonnegative factorizations  of nonnegative matrices \cite{Yannakakis1991} to arbitrary closed convex cones.  We  generalize results
\cite[Theorem 13]{Fiorini2012},  \cite[Theorem 3]{Gouveia2013} and \cite[Theorem 3]{Yannakakis1991} to  show  the equivalence between the $K$-lift of a polyhedron and  the $K$-factorization of a slack matrix.

%
 When $C$ is a  full dimensional polyhedron containing no lines and $K\subset \R{m}$ a full dimensional polyhedral cone, the $K$-factorization of a  slack operator is identical  to the  $K$-factorization of the canonical  slack matrix of $C$.  Theorem \ref{kliftfact} can be deduced directly from  Theorem \ref{pointedcase} and \ref{pointedcone} when $C$ is full dimensional.
\begin{theorem}\label{kliftfact}
Let $K$ be a full dimensional closed convex cone in $\R{m}$. If a full dimensional polyhedron $C\subset\R{n}$ containing no lines has a proper  $K$-lift, then every  slack matrix of $C$ admits a $K$-factorization. Conversely, if some  slack matrix of $C$ has a $K$-factorization, then $C$ has a  $K$-lift.
\end{theorem}

\subsection{$K$ is a polyhedral cone}
Although we have pointed out in previous section the condition $\rpc=\pi(K\cap 0^+L)$ is not redundant and can not be deduced from the condition
$C=\pi(K\cap L)$ in general. When $C$ and $K$ are both polyhedra, (\ref{CKlift}) and (\ref{Klift}) are equivalent.
\begin{lemma}\label{changebylinear}
Let $C\subset \R{n}$ be a full dimensional polyhedron containing no lines and $K\subset \R{m}$ a full dimensional polyhedral cone, then $C$ has a $K$-lift defined by (\ref{CKlift}) if and only if it has a $K$-lift defined by (\ref{Klift}).
\end{lemma}
\begin{proof}
It is sufficient to show that if there exists an affine space $L$ and a linear map $\pi$ from $\R{m}$ to $\R{n}$ such that $C=\pi(K\cap L)$,  we will have $\rpc=\pi(K\cap 0^+L)$. It is clear that if we define $Q$ to be $K\cap L$, then $Q$ is a polyhedron. For  $\forall x \in Q$,  there exist extreme points $\alpha_1,\ldots, \alpha_t$ in $Q$ and non-zero extreme directions $\alpha_{t+1}, \ldots, \alpha_{t+s}$ in $0^+Q$ such that $x=\lambda_1\alpha_1+\cdots+\lambda_t\alpha_t+\lambda_{t+1}\alpha_{t+1}+\cdots+\lambda_{t+s}\alpha_{t+s}$, where $\lambda_1+\cdots+\lambda_t=1$ and  $\lambda_i\geq 0$ for $i=1,\ldots,t+s$. Then we have $\pi (x)=\lambda_1\pi(\alpha_1)+\cdots+\lambda_t\pi(\alpha_t)+\lambda_{t+1}\pi(\alpha_{t+1})+\cdots+\lambda_{t+s}\pi(\alpha_{t+s})$. So $\rpc$ is generated by $\pi(\alpha_{t+1}),\cdots,\pi(\alpha_{t+s})$. On the other hand, since $\alpha_{t+1},\cdots,\alpha_{t+s}$  generate $0^+Q= K\cap 0^+L$, thus $\pi(\alpha_{t+1}),\cdots,\pi(\alpha_{t+s})$ generate $\pi(K\cap 0^+L)$. Hence,  $\rpc=K\cap 0^+L$ and our proof is completed.
\end{proof}

\begin{theorem}\label{polyhedraexact}
Let $C\subset\R{n}$ be a full dimensional polyhedron containing  no lines and $K\subset\R{m}$  a full dimensional polyhedral cone. If $C$ is not a translated convex cone, then $C$ has a $K$-lift defined by (\ref{CKlift}) if and only if the  slack matrix of $C$ in Definition \ref{slacek} has a $K$-factorization defined by Definition \ref{kfactorizable}. If $C$  is a translated convex cone, we can have  the same result if we replace the  $K$-lift defined by ($\ref{CKlift}$) by the $K$-lift defined by ($\ref{conelift}$).
\end{theorem}

\begin{proof}
 According to Lemma \ref{changebylinear}, the polyhedron $C$ has a $K$-lift defined by (\ref{CKlift}) if and only if it has a  $K$-lift defined by (\ref{Klift}).
The properness condition that $L$ intersects  $\inte{K}$  is  used to guarantee the strong duality  in the proof  Theorem \ref{kliftfact} (see  the proof of Theorem 1 in \cite{Gouveia2013}).  When $K$ is a polyhedral cone,  the minimization problem involved  in the proof becomes a  linear-programming problem  and  the strong duality holds if $K \cap L\neq \emptyset$.
\end{proof}

\begin{example}\label{noncomexample} Consider the polyhedron $C\subset\R{2}$ defined by
\begin{equation*}
C=\left\{ (x_1,x_2)\in\mathbb{R}^2:
\left( \begin {array}{cc} 0&1\\ \noalign{\medskip}-2+\sqrt {3}&1
\\ \noalign{\medskip}1-\sqrt {3}&\sqrt {3}-1\\ \noalign{\medskip}-1&2-
\sqrt {3}\\ \noalign{\medskip}-1&-2+\sqrt {3}\\ \noalign{\medskip}1-
\sqrt {3}&1-\sqrt {3}\\ \noalign{\medskip}-2+\sqrt {3}&-1
\\ \noalign{\medskip}0&-1\end {array}
\right)
\left(
  \begin{array}{c}
    x_1\\
    x_2
  \end{array}
\right)
\leq
 \left( \begin {array}{c} 2\\ \noalign{\medskip}\sqrt {3}
\\ \noalign{\medskip}1\\ \noalign{\medskip}2-\sqrt {3}
\\ \noalign{\medskip}\sqrt {3}-2\\ \noalign{\medskip}-2\,\sqrt {3}+3
\\ \noalign{\medskip}\sqrt {3}-2\\ \noalign{\medskip}0\end {array}
 \right)
\right\}.
\end{equation*}
\begin{figure}
\includegraphics[width=0.6\textwidth]{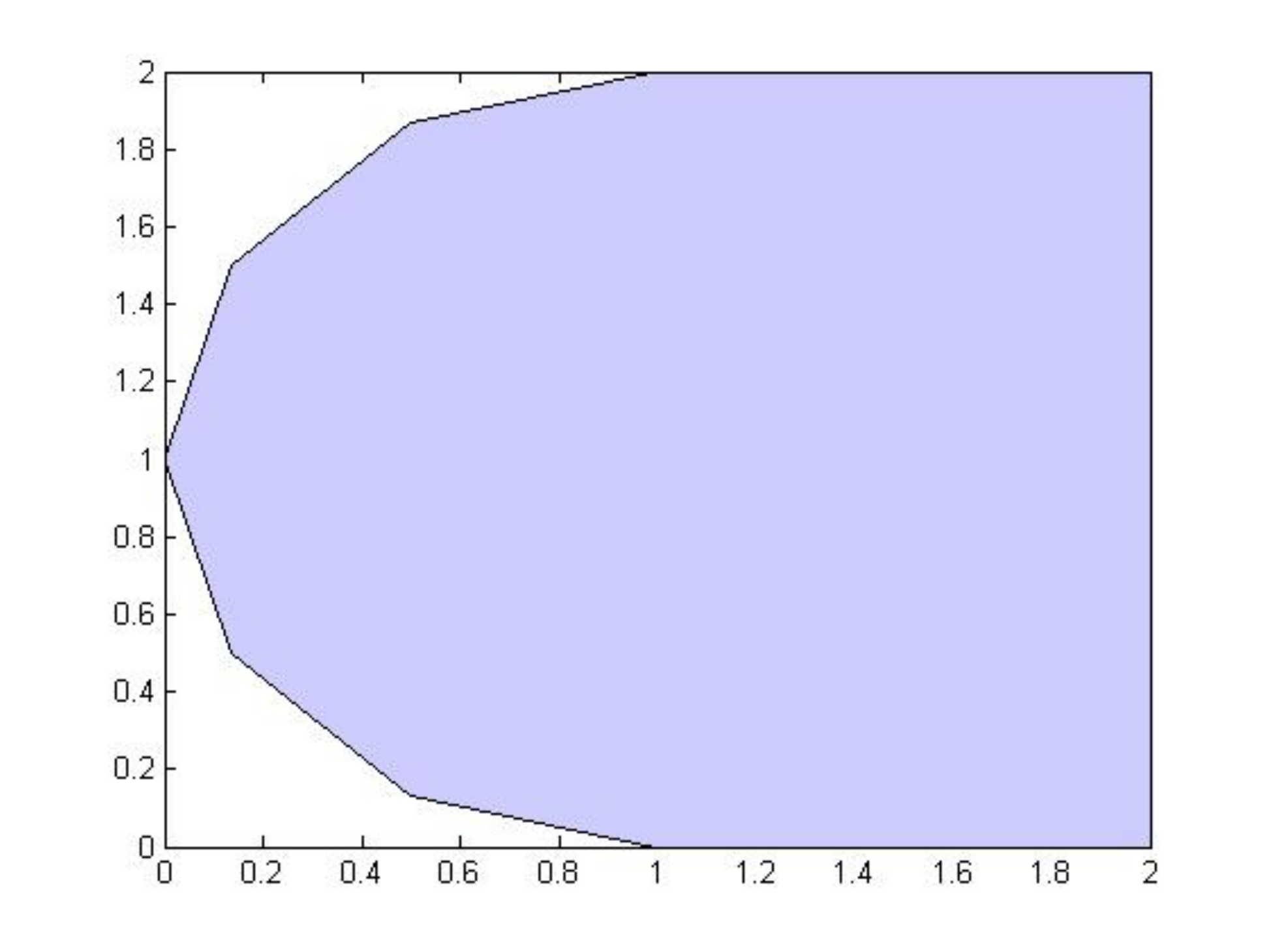}
\caption{Example \ref{noncomexample}}
\end{figure}
By Theorem \ref{polyhedraexact}, $C$ has a $\mathbb{R}^6_+$-lift if and only if the slack matrix $S$ has a $\mathbb{R}^6_+$-factorization. We denote the coefficient matrix by $H$ and the right hand side vector by $d$.
The slack matrix $S$ is
\begin{equation*}
 S=\left[ \begin {array}{cccccccc} 0&1-\frac{\sqrt {3}}{2}&\frac{1}{2}&1&\frac{3}{2}&1+\frac{\sqrt {3}}{2}
 &2&0\\ \noalign{\medskip}0&0&2-\sqrt {3}&\sqrt {3}-1&3-\sqrt
{3}&\sqrt {3}&2&2-\sqrt {3}\\ \noalign{\medskip}2-\sqrt {3}&0&0&2-
\sqrt {3}&\sqrt {3}-1&3-\sqrt {3}&\sqrt {3}&\sqrt {3}-1
\\ \noalign{\medskip}\sqrt {3}-1&2-\sqrt {3}&0&0&2-\sqrt {3}&\sqrt {3}
-1&3-\sqrt {3}&1\\ \noalign{\medskip}3-\sqrt {3}&\sqrt {3}-1&2-\sqrt {
3}&0&0&2-\sqrt {3}&\sqrt {3}-1&1\\ \noalign{\medskip}\sqrt {3}&3-
\sqrt {3}&\sqrt {3}-1&2-\sqrt {3}&0&0&2-\sqrt {3}&\sqrt {3}-1
\\ \noalign{\medskip}2&\sqrt {3}&3-\sqrt {3}&\sqrt {3}-1&2-\sqrt {3}&0
&0&2-\sqrt {3}\\
 \noalign{\medskip}2&1+\frac{\sqrt {3}}{2}&\frac{3}{2}&1&\frac{1}{2}&1-\frac{\sqrt {3}}{2}&0&0\end {array} \right].
\end{equation*}
We compute  a $\mathbb{R}^6_+$-factorization of  $S$ as $S=U\cdot V$ where
\begin{align*}
U= \left[ \begin {array}{cccccc} 1&1&0&1-\frac{\sqrt {3}}{2}&0&0
\\ \noalign{\medskip}1&-2\,\sqrt {3}+4&2-\sqrt {3}&0&0&0
\\ \noalign{\medskip}\sqrt {3}-1&0&\frac{\sqrt {3}}{2}-\frac{1}{2}&0&0&\frac{\sqrt {3}}{2}-\frac{1}{2}
\\ \noalign{\medskip}2-\sqrt {3}&0&0&2-\sqrt {3}&0&1
\\ \noalign{\medskip}0&0&0&2-\sqrt {3}&2-\sqrt {3}&1
\\ \noalign{\medskip}0&0&\frac{\sqrt {3}}{2}-\frac{1}{2}&0&\sqrt {3}-1&\frac{\sqrt {3}}{2}-\frac{1}{2}\\ \noalign{\medskip}0&-2\,\sqrt {3}+4&2-\sqrt {3}&0&1&0
\\ \noalign{\medskip}0&1&0&1-\frac{\sqrt {3}}{2}&1&0\end {array} \right],
\end{align*}
\begin{align*}
V=\left[ \begin {array}{cccccccc} 0&0&0&0&1&\sqrt {3}&2&0
\\ \noalign{\medskip}0&0&\frac{1}{2}&1&\frac{1}{2}&0&0&0\\ \noalign{\medskip}0&0&0&
\sqrt{3}-1&0&0&0&1\\ \noalign{\medskip}0&1&0
&0&0&1&0&0\\ \noalign{\medskip}2&\sqrt {3}&1&0&0&0&0&0
\\ \noalign{\medskip}\sqrt {3}-1&0&0&0&0&0&\sqrt {3}-1&1\end {array}
 \right].
\end{align*}
The $\mathbb{R}^6_+$-lift of  $C$ is:
\begin{equation*}
C=\{(x_1,x_2)\mid \exists y\in\mathbb{R}_+^6 ~{\rm s.t.}~ \ Hx+ U y=d\}.
\end{equation*}
If we eliminate $x_1$ and $x_2$ from $Hx+ U y=d$, we  have
\begin{align*}
\{y\in \mathbb{R}_+^6\mid y_{{1}}&=1+(\sqrt {3}-1)y_{4}+\frac{\sqrt {3}+1}{2}\,y_{6}-\frac{\sqrt {3}+1}{2}\,y_{3}
-y_{{5}},\\
y_{{2}}&=\frac{1}{2}-\frac{1}{2}\,y_{{4
}}-\frac{\sqrt {3}+1}{4}\,y_{{6}}+\frac{\sqrt {3}+1}{4}\,y_{{3}}\}.
\end{align*}
\end{example}

\subsection{$K$ is a positive semidefinite cone}

The {\itshape positive semidefinite rank} of a polytope  $C$ is the smallest $k$ such that $C$ has an $\mathcal S_{+}^k$-lift  \cite{Gouveia2013, GRT2013}.  A lower bound on the psd rank of a polytope is given in \cite[Proposition 3.2]{GRT2013}. Now we  extend this result to the  case where $C$ is a polyhedron. The following lemma extends the result in \cite[Proposition 3.8]{GRT2013}.

\begin{lemma}\label{psdplusone}
  Assume $C$ is a full dimensional polyhedron containing no lines. The polyhedron $C\subset\R{n}$ has a facet of psd rank $k$, then the psd rank of $C$ is at least $k+1$.
\end{lemma}
\begin{proof}
Let $F$ be a facet of $C$. Assume  the slack matrix $S_F$ of $F$  has psd rank $k$. Let   $\alpha_1, \ldots, \alpha_s$ be vertices of $F$ and $\alpha_{s+1}, \ldots, \alpha_{s+t}$ extreme directions of $F$. Suppose the facets of $F$ correspond to the facets  $F_1, \ldots, F_r$ of $C$ other than $F$.
Since $F \neq C$, there exists a vertex or an extreme direction denoted by $\alpha$ which does not belong to $F$ and $F(\alpha)>0$.  The  slack matrix $S_C$ of $C$ contains a  $(r+1)\times (s+t+1)$ submatrix which
is indexed by $F_1,\ldots,F_r,F$ in the row and $\alpha_1,\ldots,\alpha_s,\alpha_{s+1},\ldots,\alpha_{s+t}, \alpha$ in the column and has the following form
\begin{equation*}
S'=\left(\begin{array}{cc}
S_F&w\\
0& F(\alpha)
\end{array}
\right) ~{\text{where}}~ w \in \RR_{+}^r \text{ and } F(\alpha)>0.
\end{equation*}
According to  \cite[Proposition 2.6]{GRT2013}, we know that the psd rank of $S'$ is $k+1$. Hence $S_C$ has psd rank at least $k+1$.
\end{proof}

\begin{theorem}
If $C\subset \R{n}$ is a full dimensional polyhedron  that contains no lines, then the psd rank of $C$ is at least $n$.
\end{theorem}

\begin{proof}
 The proof is very similar to the one given in \cite[Proposition 3.2]{GRT2013}.  The only difference is that if   $n=1$, $C$ can be  a half line. Hence  there exists a  slack matrix whose size is $1$ by $2$. Obvious, the psd rank of this slack matrix is $1$.
Assume the statement holds up to dimension $n-1$. We select a facet $F$ of $C$ which has dimension $n-1$ and its psd rank is at least $n-1$. By Lemma \ref{psdplusone}, the psd rank of $C$ is at least $n$.
\end{proof}

\begin{remark}
There exists a full dimensional polyhedron $C\subset\R{n}$ that contains no lines such that the psd rank of $C$ is $n$. For example, consider the $n$-dimensional nonnegative orthant $\RR_{+}^{n}=\{x\mid x_i\geq 0,\ i=1,\ldots, n\}$. The  slack matrix of $\RR_{+}^{n}$ is $\left( {\sf{0}},I_n\right)$ where $I_n$ is a unit matrix and ${\sf{0}}$ is a zero vector and its psd rank is $n$.
\end{remark}

\subsection{Identifying the slack matrix of a polyhedron}
Gouveia et al. in  {\cite{Gouveia20132921} purposed algorithmic methods to identify whether a  nonnegative matrix is a slack matrix of a polyhedral cone or a polytope.  These results can be generalized to characterize the  slack matrix of a polyhedron. Similar to \cite[Lemma 10]{Gouveia20132921}, we have the following lemma:
\begin{lemma}\label{slacksame}
A nonnegative matrix $S$ is a  slack matrix of a polyhedron $C$ if and only if it is a  slack matrix of a full dimensional polyhedron which contains no lines.
\end{lemma}
\begin{proof}
 Since the  slack matrix of a polyhedron is also a slack matrix of its translation. We can assume that the polyhedron  contains the origin.
  Assume that  $C$ contains lines, and  can be decomposed as $C=\pointed+\linc$, where   $\pointed=C\cap \linco$ is a convex set containing  no lines and $\linco$ is the orthogonal complement of $\linc$.

Let $C$ be a polyhedron defined by a set of  linear inequalities $f_i(x) \leq\alpha_{i}, \ g_j(x) \leq 0$
where $\alpha_{i} >0$, $1 \leq i \leq k_1$ and $1 \leq j \leq k_2$.  Every point in $C$ can be expressed by the convex combination of a set of  points $c_1,\ldots,c_t$ and  directions $r_1,\ldots,r_s$.  According to Definition \ref{slacek},  the  slack matrix $S$ of a polyhedron $C$ can be factorized  as
\begin{align}\label{slackmatrix}
S=U \cdot V =\left(
  \begin{array}{cc}
    \alpha_1 & -f_1\\
    \vdots & \vdots\\
    \alpha_{k_1}& -f_{k_1}\\
    0 & -g_1\\
    \vdots & \vdots\\
    0& -g_{k_2}\\
  \end{array}
\right)
\left(
  \begin{array}{cccccc}
    1& \cdots& 1& 0& \cdots& 0\\
    c_1& \cdots&c_t& r_1&\cdots&r_s
  \end{array}
\right).
\end{align}
Since  the  linear functions corresponding to $f_i$  and   $g_j$ are  bounded above on $C$, $f_i$  and   $g_j$ are  orthogonal to $\linc$. Let $Q$ be the orthogonal basis of $\linco$, then we have $f_i \cdot (I-QQ^T)=\sf{0}$ and $g_j\cdot (I-QQ^T)=\sf{0}$ where $I$ is an identity matrix and $\sf{0}$ is a zero vector.  We have the following equalities:
\begin{align*}\label{transformatrix}
 S&=U \cdot \left(\begin{array}{cc}1&0\\ 0& I-QQ^T+QQ^T\end{array}\right) \cdot V \\
 &=U\cdot\left(\begin{array}{cc}0&0\\ 0& I-QQ^T\end{array}\right) \cdot V+U\cdot\left(\begin{array}{cc}1&0\\ 0& QQ^T\end{array}\right) \cdot V\\
&=U\cdot\left(\begin{array}{cc}1&0\\ 0& Q\end{array}\right)\cdot\left(\begin{array}{cc}1&0\\ 0& Q^T\end{array}\right)\cdot V.
\end{align*}
Let $U'=U\cdot\left(\begin{array}{cc}1&0\\ 0& Q\end{array}\right)$ and $V'=\left(\begin{array}{cc}1&0\\ 0& Q^T\end{array}\right)\cdot V$, it is easy to see that $S=U' \cdot V'$ is the  slack matrix of  $Q^T C_0$, which is a  polyhedron that contains no lines.
%

 If $\pointed$ is not full dimensional, $\aff{\pointed}$ is a nontrivial linear space. By similar  transformations used above, we can show $S$ is the slack matrix of $\pointed$ in $\aff{\pointed}$.
\end{proof}

The following theorem and its proof is similar to  \cite[Theorem 6]{Gouveia20132921}.

\begin{theorem}
A nonnegative matrix $S\in\mathbb{R}_+^{p\times q}$ with $\text{rank}(S)\geq 2$ is a slack matrix of a polyhedron  if and only if $S$ is a slack matrix of a polyhedral cone and there exists a vector  whose component consists of only $0$ and $1$ contained in the row space of $S$.
\end{theorem}
%
%

In   \cite[Theorem 14]{Gouveia20132921},  \cite[Corollary 5]{Gillis2012} and  \cite[Lemma 3.1]{GRT2013}, they characterized
the rank of a slack matrix in terms of the dimension of  a polytope. When $C$ is a pointed polyhedral cone,  its dimension is equal to the rank of its slack matrix  \cite[Lemma 13]{Gouveia20132921}).  These  results can be extended to the case that  $C$ is  a polyhedron.

\begin{theorem}\label{rankbound}
Let $C\subset \R{n}$ be a $n$-dimensional  polyhedron containing no lines. If $C$ is not a translated convex cone, then  the rank of the  slack matrix $S$ is $n+1$.
\end{theorem}
\begin{proof}
Suppose $C$ is not a translated convex cone.  Since  the rank of its  slack matrix does not change after the translation of $C$ and all the slack matrices of $C$ have the same rank, we can assume that $C$ contains the origin and  its canonical slack matrix  can be written as (\ref{slackmatrix}).
We  show that   the matrix $U$ is of full  column rank.  Otherwise,
 there exists  a vector $\left(\begin{array}{l}x_1\\ x_2 \end{array}\right)$ such that $U \cdot \left(\begin{array}{l}x_1\\ x_2 \end{array}\right)=0$.
 If $x_1\neq 0$, set  $x_1=1$.
 Then  $1-f(x_2)=0$ for all $f\in D_1$ and  $g(x_2)=0$ for all $g\in D_2$. By Theorem \ref{identify one extreme}, $C$ is a translated convex cone. This leads to a contradiction.  If $x_1=0$, since for every vector $y\in \cp$, there exist $\lambda_i^1\geq0,\  1\leq i\leq k_1$ and $\lambda_j^2\geq0,\  1\leq j \leq k_2$  such that $y=\sum\limits_{i=1}^{k_1}\lambda_i^1 f_i+\sum\limits_{j=1}^{k_2}\lambda_j^2 g_j$ for $f_i\in D_1$ and $g_j \in D_2$, we have $\langle x_2,y\rangle=0$. Since $\dim(\cp)=n$,  $\cp$ contains an interior. We derive   $x_2=0$ since $\langle y,x_2\rangle=0$ for each $y$ in $\cp$. Hence, $U$ is a full column rank matrix.  Moreover, since $C$ is a $n$-dimensional polyhedron, the dimension of the cone in $\R{n+1}$ generated by vectors $\left(\begin{array}{l}1\\ c_1 \end{array}\right),\ldots, \left(\begin{array}{l}1\\ c_t \end{array}\right), \left(\begin{array}{l}0\\ r_1 \end{array}\right),  \ldots, \left(\begin{array}{l}0\\ r_s \end{array}\right)$ is $n+1$. Hence, the matrix $V$ has rank $n+1$.  Therefore, the rank of $S$ is $n+1$.
\end{proof}

\begin{corollary}\label{rankthree}
Let $C$ be a polyhedron such that $C=\pointed+\linc$ where $\linc$ is the lineality space of $C$ and $\pointed=C\cap \linco$. If $C$ is not a translated convex cone, then  the rank of its  slack matrix $S$ is $\text{dim}(\pointed)+1$.
\end{corollary}

In \cite[Theorem 3.2]{Shitov2014}, they  gave an upper bound $\lceil 6\min\{m,n\}/7 \rceil$ of the nonnegative rank for a rank-three nonnegative matrix in $ \R{m\times n}$. By Theorem \ref{rankbound} and  Corollary \ref{rankthree},
  the  slack matrix $S\in \R{m\times n}$ of every polyhedron in $\R{2}$ is rank-three except that the polyhedron is a translated convex cone. It is easy to show that when $\min(m,n)\geq 7$, every  such slack matrix has a nontrivial nonnegative factorization. This fact motivates us to   compute a  $\mathbb{R}^6_+$-factorization of the slack matrix $S \in \R{8\times 8}$ in  Example \ref{noncomexample}.

\bigskip
\noindent {\bf Acknowledgement} Chu
 Wang and Lihong Zhi are partially supported
by National Key Basic Research
Project  NKBRPC 2011CB302400, and the Chinese National Natural Science
Foundation under grant NSFC 91118001, 60911130369. This work is supported by National Institute for Mathematical
Sciences 2014 Thematic Program on Applied Algebraic Geometry in Daejeon,
South Korea.

\def\refname{\Large\bfseries References}
\bibliographystyle{plain}

\end{document}